\documentclass[11pt]{amsart}
\usepackage{amsmath}
\usepackage{amssymb}
\usepackage{tabularx}
\usepackage{enumerate}
\usepackage{graphicx}
\usepackage{texdraw}
\usepackage{color}

\usepackage{mathrsfs}

\usepackage{amsfonts,amssymb,amsmath}
\usepackage{epsfig}

\makeatletter
\@addtoreset{equation}{section}

\makeatother
\newtheorem{thm}{Theorem}[section]
\newtheorem{lem}[thm]{Lemma}

\newtheorem{prop}[thm]{Proposition}

\newcommand{\R}{\mathbb{R}}
\newcommand{\ve}{\varepsilon}
\newcommand{\wt}{\widetilde}

\newcommand{\vi}{\varphi}
\begin{document}
\title[Translating Solutions of a Mean Curvature Flow]
      {Translating Solutions of a Generalized Mean Curvature Flow in a Cylinder: I. Constant Boundary Angles$^*$}
\thanks{$^*$ This research was partly supported by Natural Science Foundation of China (No. 12071299, No. 12001375).}

\author[Bendong Lou and Lixia Yuan]{Bendong Lou$^{\dag}$ and Lixia Yuan$^{\dag, \ddag}$}
\thanks{$\dag$ Mathematics and Science College, Shanghai Normal University, Shanghai 200234, China.}
\thanks{{\bf Emails:}  {\sf lou@shnu.edu.cn} (B. Lou),  {\sf yuan\underline{\ }shnu@hotmail.com} (L. Yuan)}
\thanks{$\ddag$ Corresponding author.}
\date{}


\maketitle

\begin{abstract}
We study a generalized mean curvature flow involving a positive power of the mean curvature and a driving force. In this paper, we first construct {\it all kinds} of radially symmetric translating solutions, and then select one of them to satisfy a prescribed boundary angle in a cylinder. We then
consider the flow starting at an initial hypersurface: showing the a priori estimates (especially the {\it uniform-in-time bounds} for the mean curvature which guarantee the uniform parabolicity of the corresponding fully nonlinear equation), giving the global existence for the solution of the initial boundary value problem, and proving its convergence to the corresponding translating solution. Our study provides a complete exposition on the influence of the dimension, the power of the mean curvature, the driving force and the boundary angles on the existence and stability of radially symmetric translating solutions.
\end{abstract}

%
%
%
%
%
%
%


\section{Introduction}
Consider the following mean curvature flow
\begin{equation}\label{mcf1}
\frac{\partial {\gamma}(t)}{\partial t}= (H^\alpha + b) {\bf n},
\end{equation}
where $\alpha>0$ and $b\in \R$, ${\gamma}(t)$ is a moving smooth hypersurface in $\R^{N+1}$, $H$ denotes its mean curvature, and ${\bf n}$ denotes its unit normal vector.
In the special case where ${\gamma}(t)$ is a graphic hypersurface in a cylinder
$D\times \R$ for $D\subset \R^N$, that is, $ {\gamma}(t)$  is the graph of a function  $x_{N+1}=u(x,t),\ x=(x_1,\cdots, x_N)\in D$, then we have
\begin{equation}
{\bf n}=\frac{(-Du,1)}{\sqrt{1+|Du|^2}},\quad H=\mathrm{div}\, {\bf n},
\end{equation}
with $Du:=(D_1 u, D_2u, \cdots, D_Nu)$ and $D_i u=\frac{\partial u}{\partial x_i}\ (i=1,2,\cdots,N)$. Here we assume ${\bf n}$ is the upward unit normal vector. In this case, the equation (\ref{mcf1}) is converted into
\begin{equation}\label{eq0}
\frac{u_t}{\sqrt{1+|Du|^2}}  =  (\mathrm{div}\, {\bf n})^\alpha+b = \left[\left(\delta_{ij}-\frac{D_iuD_ju}{1+|Du|^2}\right)\frac{D_{ij} u}{\sqrt{1+|Du|^2}}  \right]^\alpha + b,\quad x\in D,\
 t>0.
\end{equation}
This is a fully nonlinear parabolic equation when $\alpha\not= 1$ and $H^{\alpha-1}>0$, and a quasilinear one when $\alpha =1$. For convenience, we call the former as a {\it generalized mean curvature flow} (GMCF, for short), and the latter as a {\it mean curvature flow} (MCF, for short).
When $D$ is a bounded domain in $\R^N$, this equation is often equipped with prescribed boundary angles as boundary conditions:
\begin{equation}\label{bdry0}
{\bf n}\cdot (\nu,0) = \frac{-Du\cdot \nu}{\sqrt{1+|Du|^2}} =g(x,t,u(x,t)),\quad x\in \partial D,\ t>0,
\end{equation}
where $\nu$ denotes the inner unit normal vector on the boundary $\partial D$ of $D$, $g$ is a function satisfying $|g|<1$, and $\theta:= \arccos (g)$ denotes the contact angle between $\gamma(t)$ and $\partial D\times \R$, which is called a prescribed boundary angle.

If $g$ is independent of $t$ and $u$, the problem \eqref{eq0}-\eqref{bdry0} may have a special solution of the form
\begin{equation}\label{ts1}
u(x,t)=\varphi(x)+ct,
\end{equation}
which is called a {\it translation solution} (TS, for short) or a {\it traveling wave solution} (TW, for short), $\varphi(x),\ c\in \R$ are called the {\it profile} and the {\it speed} of the TS, respectively. This kind of solutions are of special importance since they are generally asymptotically stable, and so attract the solutions of the initial or initial-boundary value problems. We now recall some known results on the TSs of MCFs (only for the case $\alpha=1$).

\medskip
\noindent
1. {\it The cases without driving forces: $b=0$.}
In 1993, Altschuler and Wu \cite{AW1} studied the problem with $N=\alpha=1,\ b=0$, that is, the problem
\begin{equation}\label{GR-p}
 \left\{
 \begin{array}{ll}
 \displaystyle u_t = \frac{u_{xx}}{1+u_x^2}, & x\in (-1,1),\ t>0,\\
 u_x(\pm 1, t) = \pm g_0, & t>0,
 \end{array}
 \right.
\end{equation}
for some positive constant $g_0$. They proved that any global solution of \eqref{GR-p} with some initial data converges to the unique (up to a shift) TS $\varphi_{0}(x) +c_{0} t$, where
\begin{equation}\label{GR-def}
\varphi_{0}(x) := -\frac{1}{c_{0}} \ln [\cos (c_{0}x)],\quad c_{0}:= \arctan g_0,
\end{equation}
which is called {\it the grim reaper}. In higher space dimension, Huisken \cite{Huisken} proved in 1989 that, in case $N\geq 1,\ g=0$, the solution of \eqref{eq0}-\eqref{bdry0} converges to a minimal surface. In 1994, Altschuler and Wu \cite{AW2} extended their result in \cite{AW1} to two dimension case: when $N=2$, $D$ is strictly convex, $g=g(x)$ with $|D_T g|$ small, a solution either converges to a minimal surface or to a TS. Recently, Ma, Wang and Wei \cite{MWW} derived uniform gradient estimates for the problem with $N\geq 1$, $g=g(x)$ and $D$ being strictly convex, and proved that the solution converges to a TS. 

\medskip
\noindent
2. {\it The cases with driving forces: $b\not=0$.}
Mean curvature flows with driving forces arise in the study of scroll waves in excitable media (cf. \cite{KT}),
and in the singular limits of partial differential equations such as the Allen-Cahn equation (cf. \cite{AHM,HKMN}), where the driving force represents the difference of the stability between two stable phases; it is also a topic in geometry analysis (cf. \cite{Ang1, Ang2, ChouZhu} etc.).
The TSs of such equations has also been studied in the last decades. For example, in 2000-2001, Ninomiya and Taniguchi \cite{NT1,NT2} consider the problem \eqref{eq0}-\eqref{bdry0} with $N=\alpha=1$ and $b>0$, they constructed a TS with V-shaped profile, and specified its shape precisely. In 2006, Nara and Taniguchi \cite{Nara-Tani} further studied the stability of this TS. In 2009, Lou \cite{Lou1} also considered the case $N=\alpha=1$. For any given $b\in \R$, he constructed all possible TSs, gave the formulas of their profiles. 

\medskip
\noindent
3. {\it MCFs in a cylinder with non-constant boundary angles.} In case the boundary angle $\arccos g$ is spatially and/or temporally heterogeneous, that is, $g$ depends on $t$ and/or $u$, the problem no longer has TSs as in \eqref{ts1}. Instead, it may have TSs with changing profiles. For example, in 2006 and 2013, Matano et al. \cite{LMN, MNL} considered the equation \eqref{eq0} with $N=\alpha=1,\ b>0$ in a band domain with undulating boundaries. Their problem, in some sense, is equivalent to \eqref{eq0}-\eqref{bdry0} with $N=\alpha=1,\ g=g(u)$ being (almost) periodic. They constructed the so-called {\it (almost) periodic traveling wave solution}: $u=\varphi(x,t)+ct$ for some (almost) periodic profile $\varphi(x,t)$, studied its uniqueness and asymptotic stability. Later, Cai and Lou \cite{CaiLou2} etc. obtained analogues for the problems with $g$ being almost periodic in $t$. To distinguish the generalized (almost) periodic TWs from that in \eqref{ts1}, we will call $\varphi(x)+ct$ as a {\it classical TS}.

Another problem about the MCFs in a cylinder is the problem with unbounded boundary slopes. In the works we mentioned above, the boundary slopes $Du\cdot \nu$ are all bounded. In 2012, Chou and Wang \cite{CW} considered the equation in \eqref{GR-p} with Robin boundary conditions like
\begin{equation}\label{Robin-cond}
u_x (1,t)= u(1,t), \quad u_x(-1,t)=-u(-1,t),\quad  t>0,
\end{equation}
that is, the boundary slopes can be unbounded when $u\to \infty$. In this case, a priori gradient estimates depending on $t$ are possible and so the global existence of the solutions can be obtained in a standard way. However, the uniform-in-time a priori gradient estimates is impossible, although it is necessary in the study for the convergence of the profile of a global solution. Recently, we revisited this problem in \cite{LWY}. We used the zero number argument to give a uniform interior gradient estimates, and then proved the convergence of any solution to the grim reaper in the open interval $(-1,1)$.

\medskip
TSs of generalized mean curvature flows, such as
\eqref{mcf1} and \eqref{eq0} with $\alpha\not=1,\ b\not= 0$, seems not well studied till now. In a series of papers, we want to use the equation \eqref{mcf1} and \eqref{eq0} to  provide a complete exposition on the influence of the dimension $N$, the power $\alpha$, the driving force $b$ and the boundary angle $g$ on the profiles, the propagating speeds and the stability of the corresponding TSs. More precisely, in this series, we will consider the following cases respectively:
\begin{itemize}
\item constant boundary angles;
\item Robin boundary conditions as in \eqref{Robin-cond} (that is, unbounded boundary slopes);
\item almost periodic boundary angles an almost periodic TSs;
\item general inhomogeneous boundary conditions for non-radially symmetric solutions. 
\end{itemize}
This is the first paper of the series. In this paper we first construct {\it all possible} radially symmetric TSs of \eqref{eq0} for any $b\in \R$ (which will  be used frequently in this series for comparison), and then use the TSs to study the existence, uniqueness and the asymptotical behavior of the flows, starting at some initial data, in a cylinder satisfying \eqref{bdry0} with constant boundary angles.

\medskip

Throughout this paper we assume $N>1$ (the plane curves when $N=1$ are simpler and can be considered in a similar way as this paper). We mainly consider the radially symmetric solutions: $u=u(|x|,t)=u(r,t)$. In this case, the equation \eqref{eq0} is rewritten as
\begin{equation}\label{E}
u_t = (H^\alpha +b)\sqrt{1+u^2_r} =\left(\frac{u_{rr}}{1+u^2_r}+\frac{N-1}{r}u_r\right)^\alpha
\left(1+u^2_r\right)^{\frac{1-\alpha}{2}}+b\sqrt{1+u^2_r} ,
\ \quad r>0, \  t>0,
\end{equation}
where
\begin{equation}\label{def-H(r,t)}
H(r,t):= \frac{u_{rr}}{(1+u^2_r)^{3/2}} +\frac{(N-1)u_r}{r\sqrt{1+u^2_r}}
\end{equation}
is the mean curvature. When we consider the problem in the cylinder
$$
\Omega := B\times \R\quad \mbox{with}\quad B:= \{x\in \R^N \mid |x|\leq 1\},
$$
the radially symmetry and the boundary condition \eqref{bdry0} are expressed as
\begin{equation}\label{BC}
u_r(0,t)=0,\quad u_r (1,t) = k(t,u(1,t)):= \frac{g(t,u)}{\sqrt{1-g^2(t,u)}},\quad t>0.
\end{equation}
We consider  in this paper the homogeneous boundary conditions, that is, the case $k$ is a constant. We will study the global existence of solution $u(r,t)$ of \eqref{E}-\eqref{BC} with initial data $u(r,0)=u_0(r)$ satisfying
\begin{equation}\label{cond-IC}
u_0(r)\in C^2 ([0,1]),\quad u'_0(0)=0,\quad u'_0(1)=k\in \R,
\end{equation}
as well as its convergence to the corresponding TS.

Since the equation is a fully nonlinear one when $\alpha\not=1$,
in order to study the global existence and the convergence of $u$, we need not only the uniform-in-time $C^2$ a priori estimates for $u$, but also the uniform-in-time parabolicity for the equation. More precisely, when we consider a flow with $H>0$, we need to show that $H$ has {\it positive lower and upper bounds}, or in the case where $H<0$, we need to show that $H$ has {\it negative lower and upper bounds}. These estimates are crucial in our approach. They turn out to be true under suitable assumptions on the parameters $b,k$ and the initial data $u_0$.

\begin{thm}\label{thm:est}
Let $u(r,t)$ be a classical solution of \eqref{E}-\eqref{BC}-\eqref{cond-IC} in the time interval $[0,T)$, then there exist positive numbers $M_0, M_1, M_2, V_*, V^*, H_*, H^*$, independent of $t$ and $T$, and $\tilde{c}(b,k)>0$ such that
\begin{eqnarray}
|u(r,t) - \tilde{c}(b,k)t|\leq M_0,\quad r\in [0,1],\ t\in [0,T);\\
0< u_r(r,t) \cdot {\rm sgn}(k) \leq  M_1,\quad r\in (0,1],\ t\in [0,T);\\
V_* \leq u_t(r,t) \leq V^*,\quad r\in [0,1],\ t\in [0,T);\\
H_* \leq H(r,t) \cdot {\rm sgn}(k)\leq H^*, \quad r\in [0,1],\ t\in [0,T);\\
|u_{rr}(r,t)|\leq M_2,\quad r\in [0,1],\ t\in [0,T),
\end{eqnarray}
provided the constants $b,k$ and the initial data $u_0$ satisfy one of the following conditions
\begin{enumerate}[{\rm (A).}]
\item $b=0<k$ and $H (r,0) >0 \ \ (r\in [0,1])$;

\item $b<0<k$, $(-b)^{{1/\alpha}} \sqrt{1+k^2}< k N$ and $H^\alpha(r,0) +b >0 \ \ (r\in [0,1])$;

\item $b>0,\ k>0$ and $u''_0(r)>0\ \ (r\in [0,1])$;

\item $\alpha=\frac{q}{p}$ for some positive odd number $p$ and $q$, $b>0>k$, $b^{{1/\alpha}} \sqrt{1+k^2}> -k N$, $u''_0(r)<0$ and $ H^\alpha(r,0)+b>0\ \ (r\in [0,1])$.
\end{enumerate}
\end{thm}

In case (D) we actually consider solutions with negative curvature: $H<0$, so in order for $H^\alpha$ to make sense and for the equation to be a parabolic one: $H^{\alpha-1}>0$, we consider only the case where $\alpha=\frac{q}{p}$ for positive odd $p,q$ (see more in Subsection \ref{subsec:TSb>c>0}).
As we will see in Subsections 4.1 and 5.1, in the cases (A) and (B) where $b\leq 0$, the above estimates can be derived from the equation \eqref{E} and the comparison principle. In the case (C) where $H>0$ and $b>0$, however, the derivation for the positive lower bound of $H$ is not so easy, which includes several steps: to derive a positive lower bound for $u_{rr}(0,t)$ by using the zero number argument (Lemma 5.2), to show $u$ is steeper than a TS $\Phi_1$ (Lemma 5.3), to show $u_{rr}(r,t)>0$ for all $r\in [0,1]$ (Lemma 5.4), and to derive the positive lower bound for $H$ (Proposition 5.5). (In the case (D) where $H<0$ and $b>0$, the derivation for the negative upper bound of $H$ is similar.)

Once we have the a priori estimates in hand, we can show the convergence of $u$ to the corresponding TS, which is our main result in this paper.

\begin{thm}\label{thm:stable}
Assume the hypotheses {\rm (A), (B), (C)} or {\rm (D)} in the previous theorem holds. Then
\begin{enumerate}[\rm (i)]
\item the problem \eqref{E}-\eqref{BC} has a unique (up to a shift) TS $\wt{\Phi}(r)+\tilde{c}(b,k)t$;
\item the problem \eqref{E}-\eqref{BC}-\eqref{cond-IC} has a unique time-global classical solution $u(r,t)$;
\item there exists $M\in \R$ such that
\begin{equation}\label{asy-stable}
\|u(\cdot,t)-\wt{\Phi}(\cdot)- \tilde{c}(b,k)t\|_{L^\infty([0,1])}\to M \ \ \mbox{as}\ \ t\to \infty.
\end{equation}
\end{enumerate}
\end{thm}

This theorem extends the result in \cite{AW1,AW2} to the cases where $\alpha\not= 1$ and $b\not= 0$, for radially symmetric solutions.
The convergence in (iii) is equivalent to say that the TS is asymptotically  stabile.

A TS of the equation \eqref{E} (without the restriction \eqref{BC} at this stage) has the form  $u=\varphi(r)+ct$, with $(c, \varphi)$ satisfying
\begin{equation}\label{TS}
\left\{
\begin{array}{l}
\displaystyle  c =\left(\frac{\varphi''}{1+\varphi'^2}+\frac{N-1}{r}\varphi'\right)^\alpha
\left(1+\varphi'^2_r\right)^{\frac{1-\alpha}{2}}+b\sqrt{1+\varphi'^2},
\ \quad r>0, \\
  \varphi(0)=\varphi'(0)=0.
\end{array}
\right.
\end{equation}
The additional condition $\varphi(0)=0$ is used to normalize the profile of $\varphi$.
We consider only the case $c> 0$, since when the hypersurface moves with speed $c<0$, we only need to regard the flow as one moving to the $-x_{N+1}$ direction with a positive speed. To show the existence and the uniqueness of the TS of \eqref{E}-\eqref{BC}, we first prepare the TSs of \eqref{E}, and then select some of them to satisfy the boundary condition at $r=1$. The following result presents {\it all possible} solutions of \eqref{TS}.

\begin{thm}\label{thm:TS-in-whole}
For any $b\in \R$ and any $c>0$, the problem \eqref{TS} has a unique solution $\varphi=\Phi(r;c,b)$ in a maximal existence interval $J:= [0,R_\infty(c,b))$. It is smooth in $J$, $\Phi(r;c,b)$ and $\Phi'(r;c,b)$ are strictly increasing in $c$ and $-b$. Moreover,

\begin{enumerate}[{\rm (i).}]
 \item in case $b=0$, we have $R_\infty = \infty$, $\Phi'(r),\ \Phi''(r)>0$ for $r>0$ and
 \begin{equation}\label{phi-shape-b=0}
 \Phi(r) \sim \frac{c}{(\alpha+1)(N-1)^\alpha}  r^{\alpha+1} \mbox{\ \ as\ \ } r\to \infty;
 \end{equation}

 \item in case $b<0$, we have $R_\infty(c,b)<\infty$, $\Phi'(r),\ \Phi''(r)>0$ in $(0,R_\infty)$ with  \begin{equation}\label{phi-psi-b<0}
   \Phi( R_\infty-0) <\infty,\quad  \Phi'(R_\infty-0) =   \Phi''(R_\infty-0) =\infty;
     \end{equation}

 \item in case $c>b>0$, we have $R_\infty =\infty$, $\Phi'(r),\ \Phi''(r)>0$ in $(0,\infty)$ with
 $$
 \Phi'(r)\to \Psi_0:= \frac{\sqrt{c^2-b^2}}{b}\mbox{\ \  as\ \ } r\to \infty;
 $$

 \item in case $c=b>0$, we have $R_\infty = \infty$, and $\Phi(r)\equiv 0$;

 \item in case $\alpha=\frac{q}{p}$ for some odd number $p,q>0$ and $b>c>0$, we have $R_\infty(c,b)<\infty$, $\Phi'(r),\ \Phi''(r)<0$ in $(0,R_\infty)$, and
  \begin{equation}\label{phi-psi-b>0}
  \Phi( R_\infty-0) > -\infty,\quad  \Phi'(R_\infty-0) =   \Phi''(R_\infty-0) = -\infty.
  \end{equation}
\end{enumerate}
\end{thm}

\medskip
As we have mentioned above, in case $b=0$, the profile of the TS is a grim reaper when $N=\alpha=1$. Here we present the (asymptotic) profiles of the TSs for all $\alpha >0$ as in \eqref{phi-shape-b=0}.
In the special case $\alpha=1$, our result coincides with that in \cite{CGHNR}, where the authors consider the Allen-Cahn equation and proved that the level set of its solution (which roughly obeys the MCF) is asymptotically a paraboloid
$$
\Phi(r) \sim \frac{c}{2(N-1)}|x|^2,\quad r\to \infty,
$$
(see also \cite{AW2} for the case $N=2,\ b=0,\ \alpha=1$).
In cases (ii) and (v) we have $R_\infty<\infty$ and $|\Phi(R_\infty-0)|<\infty$, so the profiles of the TSs are {\it cup-like} ones in case (ii) and  {\it cap-like} ones in case (v), as it was shown in \cite{Lou1} for the case $N=\alpha=1$.
In case (iii), the profiles of the TSs are cone-shaped ones, as the V-shaped TSs in the case $N=\alpha=1$ (\cite{NT1, NT2}). However, each of them is only an approximate but not exact cone since
$$
|\Phi(r)-\Psi_0 r|\to \infty\ (r\to \infty)
$$
in some cases (see details in Section 3).

We then select some TSs from the above candidates so that they lie in the cylinder $\Omega$, and satisfy the boundary conditions in \eqref{BC} for some $k\in \R$, that is,
\begin{equation}\label{TS-BC}
\varphi'(1) = k.
\end{equation}

\begin{thm}\label{thm:TS-in-cylinder}
Assume one of the following conditions holds:
\begin{enumerate}[{\rm (a).}]
\item $b=0<k$;
 \item $b<0<k$ and $(-b)^{{1/\alpha}} \sqrt{1+k^2}< k N$;
 \item $b>0$ and $k> 0$;
 \item $\alpha=\frac{q}{p}$ for some positive odd number $p$ and $q$, $b>0>k$, and $b^{{1/\alpha}} \sqrt{1+k^2}> -k N$.
\end{enumerate}
Then, there exists a unique $\tilde{c}=\tilde{c}(b,k)>0$ such that $(c,\varphi)=(\tilde{c},\wt{\Phi}(r))$ with $\wt{\Phi}(r):= \Phi(r;\tilde{c}(b,k),b)$ solves the problem \eqref{TS} and \eqref{TS-BC}.

In addition, the conclusions hold even if $k=\infty$ in case {\rm (b)}, or, $k=-\infty$ in case {\rm (d)}.
\end{thm}

These TSs, then, can be used to give necessary a priori estimates for $u$, and $\wt{\Phi}(r)+\tilde{c}(b,k)t$ will be the asymptotic limit of $u$ as shown in Theorem \ref{thm:stable}.

The paper is arranged as follows. In Sections 2 we construct TSs for the case $b\leq 0$. In Section 3, we construct TSs for the case $b>0$. Theorem \ref{thm:TS-in-whole} follows from Theorems \ref{thm:TSb<0}, \ref{thm:TSc>b>0} and \ref{thm:TSb>c>0}. Theorem \ref{thm:TS-in-cylinder} is proved in Subsections 2.2, 3.1 and 3.2.
In Section 4, we consider the problem \eqref{E}-\eqref{BC} with $b\leq 0$. We first give a priori estimates in Theorem \ref{thm:est-b<0}, and then prove the global existence and convergence of the solution to TS in Theorem \ref{thm:stability-TS-b<0}. In Section 5, we consider the
problem \eqref{E}-\eqref{BC} in case $b>0$. In the first subsection we study the case $k>0$. The main difficulty is to derive the positive lower bound for $H$, which follows from a series of lemmas and is given in Proposition \ref{prop:H>delta}. Then we obtain the global existence and the convergence of the solution of the initial-boundary value problem in Theorem \ref{thm:stabilty-TS-c>b>0}. In the last subsection we consider the case $b>0>k$ and present the a priori estimates, the global existence and the convergence results in Theorem \ref{thm:b>c>0}. In summary, Theorem \ref{thm:est} follows from Theorems \ref{thm:est-b<0}, Lemma \ref{lem:c>b>0-C1}, Proposition \ref{prop:H>delta} and Theorem \ref{thm:b>c>0}; Theorem \ref{thm:stable} follows from Theorems \ref{thm:stability-TS-b<0}, \ref{thm:stabilty-TS-c>b>0} and \ref{thm:b>c>0}.


\section{Translating Solutions in Case $b\leq 0$}
In this section we study the TSs in case $b\leq 0<c$. We first construct TSs of \eqref{TS}, and then select some of them to satisfy \eqref{TS-BC} for given $k>0$.

\subsection{TSs on their maximal existence intervals}
In this subsection we construct solutions of \eqref{TS} in case $b\leq 0$. Our main result to be proved is the following one.

\begin{thm}\label{thm:TSb<0}
For any $b\leq 0$ and any $c>0$, the problem \eqref{TS} has a unique solution $\varphi=\Phi(r;c,b)$ in a maximal existence interval $J:= [0,R_\infty(c,b))$. It is smooth in $J$, $\Phi(r;c,b)$ and $\Phi'(r;c,b)$ are strictly increasing in $c$ and $-b$. In addition,
\begin{enumerate}[{\rm (I).}]
 \item when $b=0$, $R_\infty = \infty$, $\Phi'(r),\ \Phi''(r)>0$ for $r>0$, \eqref{phi-shape-b=0} holds, and
\begin{equation}\label{limits-Phi'-to-infty}
\Phi'(r;c,0)\to \infty \mbox{\ \ as\ \ }c\to \infty;
\end{equation}

 \item when $b<0$, $R_\infty(c,b)$ satisfies
\begin{equation}\label{R-infty-bounds}
N(c-b)^{-1/\alpha} \leq R_\infty (c,b) \leq N (-b)^{-1/\alpha},
\end{equation}
it is continuously dependent on and strictly decreasing in $c$ with
\begin{equation}\label{limits-R-infty}
R_\infty(c,b)\to N(-b)^{-1/\alpha} \mbox{\ \ as\ \ }c\to 0,\quad R_\infty(c,b)\to 0 \mbox{\ \ as\ \ }c\to \infty.
\end{equation}
In addition, $\Phi'(r),\ \Phi''(r)>0$ in $(0,R_\infty)$, and
     \begin{equation}\label{phi-psi-b<0}
   \Phi( R_\infty-0) <\infty,\quad  \Phi'(R_\infty-0) =   \Phi''(R_\infty-0) =\infty,
     \end{equation}
and $\Phi(R_\infty(c,b)-0)$ is decreasing in $c$.
\end{enumerate}
\end{thm}

The problem \eqref{TS} is equivalent to
\begin{equation}\label{TW1}
\left\{
\begin{array}{l}
\displaystyle \frac{\varphi''}{1+\varphi'^2}+\frac{N-1}{r}\varphi'=
\left(\frac{c}{\sqrt{1+\varphi'^2}}-b\right)^{{1/\alpha}}\sqrt{1+\varphi'^2}, \quad \ r>0,\\
 \varphi(0)=\varphi'(0)=0.
\end{array}
\right.
\end{equation}
Due to the singularity of the term  $\frac{N-1}{r}$ at $r=0$, for any $\varepsilon>0$, we first consider the approximate problem
\begin{equation}\label{TW1-1}
\left\{
\begin{array}{l}
\displaystyle \frac{\varphi''}{1+\varphi'^2}+\frac{N-1}{r+\ve}\varphi'=
\left(\frac{c}{\sqrt{1+\varphi'^2}}-b\right)^{{1/\alpha}}\sqrt{1+\varphi'^2}, \quad \ r\geq 0,\\
 \varphi(0)=\varphi'(0)=0.
\end{array}
\right.
\end{equation}

\begin{lem}\label{lem2.1}
For any $c>0$, the problem \eqref{TW1-1} has a unique solution  $\varphi=\varphi_\varepsilon (r;c,b)$ in a maximal existence interval  $[0, R^\ve_\infty)$ with
\begin{equation}\label{positive-psi-psi'}
\varphi'_\ve (r),\ \varphi''_\ve (r)>0 \quad \text{in}\quad \left(0, R^\ve_\infty\right),
\quad \text{and}\quad \vi_\ve(r),\ \vi'_\ve(r) \to \infty \quad \text{as} \quad r\to R^\ve_\infty.
 \end{equation}
Moreover, when $b=0$ there holds $R^\ve_\infty =\infty$; while, when $b<0$, there holds
\begin{equation}\label{est-R-ve-infty}
R_-:=\left(c-b\right)^{-{1/\alpha}}\leq R^\ve_\infty\leq R_+:= (-b)^{-{1/\alpha}} [N+\ln 2].
\end{equation}
\end{lem}

\begin{proof}
 Set  $\psi_\ve :=\vi'_\ve$,  then \eqref{TW1-1} is converted to
\begin{equation}\label{TW1-2}
\left\{
\begin{array}{l}
\displaystyle \frac{\psi_\ve'}{1+\psi_\ve^2}=F\left(r,\psi_\ve,\ve,c,b\right):=
\left(\frac{c}{\sqrt{1+\psi_\ve^2}}-b\right)^{{1/\alpha}}\sqrt{1+\psi_\ve^2}
-\frac{N-1}{r+\ve}\psi_\ve, \qquad r\geq 0,\\
\psi_\ve(0)=0.
\end{array}
\right.
\end{equation}
Denote its maximal existence interval by $[0,R^\ve_\infty)$.
We first prove the following claim.

\medskip
\noindent
{\bf Claim 1.} $\psi_\ve(r)>0,\ \psi_\ve'(r)>0$ in $J^\ve := (0,R^\ve_\infty)$.
\medskip

In fact, due to $\psi_\ve'(0)=\left(c-b\right)^{{1/\alpha}}>0$, we have $\psi_\ve(r)>0$ for $0<r\ll 1$. If there exists $r_1>0$ such that $\psi_\ve(r_1)=0$ and $\psi_\ve(r)>0$ in $(0,r_1)$, then $\psi_\ve'(r_1)\leq 0$. But this contradicts the equation in (\ref{TW1-2}). Hence $\psi_\ve(r)>0$ in $J^\ve$. Next, by differentiating (\ref{TW1-2}) in $r$ we have
\begin{equation}  \label{psi}
  \frac{\psi_\ve''\left(1+\psi_\ve^2\right)-2\psi_\ve (\psi_\ve')^2}{\left(1+\psi_\ve^2\right)^2}
  = A(r,\psi_\ve,\ve,c,b),\quad r\geq 0,
  \end{equation}
where
$$
 A(r,\psi_\ve,\ve,c,b):= \frac{1}{\alpha}\left(\frac{c}{\sqrt{1+\psi_\ve^2}}-b\right)^{\frac{1-\alpha}{\alpha}}
  \left( \frac{c(\alpha-1)\psi_\ve\psi_\ve'}{1+\psi_\ve^2}   - \frac{b\alpha\psi_\ve\psi_\ve'}{\sqrt{1+\psi_\ve^2}} \right) -(N-1)\frac{\psi_\ve'(r+\ve)-\psi_\ve}{\left(r+\ve\right)^2}.
$$
Since  $\psi_\ve'(0)>0$, we have  $\psi_\ve'(r)>0$ for  $0<r\ll 1$ by continuity. Assume $\psi_\ve'(r)>0$ in $[0,r_2)$ and $\psi_\ve'(r_2)=0$. Then $\psi_\ve''(r_2)\leq 0$. This, however, contradicts the equation \eqref{psi} at $r_2$. This proves Claim 1.

\medskip

By Claim 1 we have three possibilities:

\begin{enumerate}[(a).]
\item $R^\ve_\infty <\infty$ and $\psi_\ve(r)\to \infty$ as $r\to R^\ve_\infty-0$;

\item $R^\ve_\infty = \infty$ and $\psi_\ve(r)\to \infty$ as $r\to \infty$;

\item $R^\ve_\infty = \infty$ and $\psi_\ve(r)\to \bar{\psi}_\ve>0$ as $r\to \infty$.
\end{enumerate}
On these possibilities we have the following claim

\medskip
\noindent
{\bf Claim 2.} (a) and \eqref{est-R-ve-infty} hold when $b<0$; (b) holds when $b=0$; (c) does not holds for any $b$.
\medskip

By contradiction we assume (c) holds. Then by the equation in \eqref{TW1-2} we have
$$
\frac{\psi_\ve'}{1+\psi_\ve^2}=F\left(r,\psi_\ve,\ve,c,b\right) \to
\left(\frac{c}{\sqrt{1+\bar{\psi}_\ve^2}}-b\right)^{{1/\alpha}}\sqrt{1+\bar{\psi}_\ve^2}
>0 \mbox{\ \ as\ \ }r\to \infty,
$$
since $c>0\geq b$. This contradicts $\psi_\ve(r)\to \bar{\psi}_\ve$.

Next we consider the case $b=0$. If $R^\ve_\infty<\infty$, then $\psi_\ve(r) \to \infty$ and $r\to R^\ve_\infty-0$, and so by the equation \eqref{TW1-2} we have
$$
\frac{\psi_\ve'}{\psi_\ve(1+\psi_\ve^2)} = c^{{1/\alpha}} (1+\psi_\ve^2)^{\frac{\alpha -1}{2\alpha}} \frac{1}{\psi_\ve}- \frac{N-1}{r+\varepsilon}\to -\frac{N-1}{R^\ve_\infty+\ve}<0 \mbox{\ \ as\ \ }r\to R^\ve_\infty-0,
$$
contradicts $\psi_\ve'(r)>0$. Therefore $R^\ve_\infty =\infty$.

Finally we consider the case $b<0$. We show that $R^\ve_\infty<\infty$ and so (a) holds in this case. Otherwise, if $R^\ve_\infty =\infty$, then for
$r\geq R_1:= 2(N-1)(-b)^{-{1/\alpha}}$, by the equation \eqref{TW1-2} we have
$$
\frac{\psi_\ve'}{1+\psi_\ve^2} > (-b)^{{1/\alpha}} \sqrt{1+\psi_\ve^2} -\frac{N-1}{r+\varepsilon}\psi_\ve \geq \left[ (-b)^{{1/\alpha}} -\frac{N-1}{r} \right] \psi_\ve \geq \frac{(-b)^{{1/\alpha}}}{2} \psi_\ve.
$$
Integrating it over $[R_1, r)$ we have
$$
\ln \frac{\psi_\ve(r)}{\sqrt{1+\psi_\ve^2(r)}} > \ln \frac{\psi_\ve(R_1)}{\sqrt{1+\psi_\ve^2(R_1)}} + \frac{(-b)^{{1/\alpha}}}{2} (r-R_1).
$$
This implies that
$$
\psi_\ve(r)\to \infty \mbox{\ \ as\ \ }r\to R_2,
$$
for some
\begin{equation}\label{def-R2}
R_2 \leq R_1 - 2(-b)^{-{1/\alpha}} \ln \frac{\psi_\ve(R_1)}{\sqrt{1+\psi_\ve^2(R_1)}}.
\end{equation}
This contradicts $R^\ve_\infty =\infty$. Furthermore, we can derive the upper and lower estimates for $R^\ve_\infty$. By \eqref{TW1-2} we have
$$
\frac{\psi_\ve'}{1+\psi_\ve^2} \leq (c-b)^{{1/\alpha}} \sqrt{1+\psi_\ve^2}.
$$
Integrating it over $[0,R^\ve_\infty)$ we obtain $R^\ve_\infty \geq R_- = (c-b)^{-{1/\alpha}}$. On the other hand, by the positivity of $\psi_\ve,\ \psi_\ve'$ and the equation \eqref{TW1-2} we have
$$
\psi_\ve'+\frac{N-1}{r}\psi_\ve  > \frac{\psi_\ve'}{1+\psi_\ve^2}+\frac{N-1}{r+\ve}\psi_\ve\geq (-b)^{{1/\alpha}},\quad r>0.
$$
Hence, $\left(r^{N-1}\psi_\ve\right)'> (-b)^{{1/\alpha}} r^{N-1}$ for $r>0$.  Integrating it over  $[0,r]$ we obtain $\psi_\ve > \frac{(-b)^{{1/\alpha}} }{N} r$. Thus
$$
\psi_\ve(r)\geq 1, \quad r\geq  R_0 := N (-b)^{-{1/\alpha}}.
$$
If $R^\ve_\infty\leq R_0$, then the upper estimate is already obtained. In what follows we assume $R^\ve_\infty >R_0$. Then in $[R_0, R^\ve_\infty)$, in a similar way as deriving \eqref{def-R2} we have
$$
R^\ve_\infty \leq R_0 - 2(-b)^{-{1/\alpha}} \ln \frac{\psi_\ve(R_0)}{\sqrt{1+\psi_\ve^2(R_0)}} \leq (-b)^{-{1/\alpha}} [N+\ln 2] =R_+.
$$
This proves Claim 2, and completes the proof of the lemma.
\end{proof}

We will take limit as $\ve$ going to $0$ for $\psi_\ve (r;c, b)$ to obtain the solution of (\ref{TW1}) rather than to \eqref{TW1-1}. For this purpose, we study the monotonicity of $\vi_\ve$  and $\psi_\ve$ in $\ve$, and their a priori estimates.

\begin{lem}\label{lem2.2}
Both $\psi_\ve$ and $\vi_\ve $ are strictly increasing in $\ve>0$.
\end{lem}

\begin{proof}
For any $0<\ve_1<\ve_2$, by the comparison principle for ODE, we can derive from the equation \eqref{TW1-2} and \eqref{psi} that
$$
0<\psi_{\ve_1}(r)<\psi_{\ve_2}(r), \quad r>0.
$$
Therefore, $\vi_{\ve_1}(r)<\vi_{\ve_2}(r)$ for $r>0$.
\end{proof}

From this lemma we see that, when $b<0$, $R^\ve_\infty \in [R_-, R_+]$ and it is decreasing in $\ve$, and so
\begin{equation}\label{def-R-infty}
R_- \leq R_\infty (c,b):=\lim_{\ve\to 0}R^\ve_\infty \leq R_+.
\end{equation}
When $b=0$, we use the same notation $R_\infty (c,b):= \lim\limits_{\ve\to 0} R^\ve_\infty$, then $R_\infty (c,0)=\infty$.

Next we give some {\it uniform} interior estimates for $\vi$.

\begin{lem}\label{lem2.3}
For any given $R\in(0,R_\infty(c,b))$, there exists $\ve_0=\ve_0(R)$ such that $[0,R]\subset[0,R^\ve_\infty)$ for all $0<\ve\leq\ve_0$. In addition,
\begin{enumerate}[{\rm (1)}]
  \item there exists $C_1=C_1(R)$ independent of $\ve$ such that
  \begin{equation}\label{phi-up}
  \|\vi_\ve (r;c,b)\|_{C^1([0,R])}\leq C_1(R);
  \end{equation}
  \item  for any given $\delta\in (0,R)$ and any positive integer $m\geq 2$, there exists $C_m=C_m(\delta,R)$ independent of $\ve$ such that
\begin{equation}\label{est-psi-k}
  \|\vi^{(m)}_\ve (r;c,b)\|_{C([\delta,R])}\leq C_m(\delta,R).
  \end{equation}
\end{enumerate}
\end{lem}

\begin{proof}
By the definition of $R_\infty$ in (\ref{def-R-infty}), there exits $\ve_0=\ve_0(R)$ such that when $0<\ve\leq\ve_0$, we have $R^\ve_\infty>R$.

(1). For any $\ve\in(0,\ve_0]$,  by Lemmas \ref{lem2.1} and \ref{lem2.2} we have
\begin{equation}\label{upper-est-psi}
  0=\psi_\ve (0) \leq \psi_\ve(r)\leq \psi_\ve(R)\leq\psi_{\ve_0}(R), \quad r\in [0,R].
\end{equation}
This implies that (\ref{phi-up}) holds.

(2). Since
$$
\frac{1}{R+\ve_0}\leq \frac{1}{r+\ve}\leq \frac{1}{\delta}, \quad r\in [\delta,R], \ \ve\in(0,\ve_0],
$$
we see that all the coefficients involving the terms $\frac{1}{r+\ve}$ in the equations of $\psi^{(m)}_\ve\ (m\geq 1)$ are bounded (uniformly in $\ve$). By iteration we obtain \eqref{est-psi-k} for all $m\geq 2$.
\end{proof}

\medskip
\noindent
{\it Proof of Theorem \ref{thm:TSb<0} {\rm (II)}}.
We prove the theorem in case $b<0$. The proof is a little complicated and is divided into several steps.

\medskip
\noindent
{\bf Step 1. Construction and monotonicity of $\Phi$}. Let $R,\ \varepsilon_0$ be that as in Lemma \ref{lem2.3}. For any sequence $\{\ve_i\}\subset(0,\ve_0]$ with $\ve_i\to 0\ (i\to\infty)$, by (\ref{phi-up}) there exist a subsequence of $\{\ve_i\}$ (denote it again by  $\{\ve_i\}$) and a function $\Phi\in C\left([0,R]\right)$ such that
\begin{equation}
\|\vi_{\ve_i}(r;c,b)-\Phi(r)\|_{C\left([0,R]\right)}\to 0 {\ \ as\ \ } i\to\infty.
\end{equation}
Moreover, for any positive integer $m$, by \eqref{est-psi-k}, there exists a subsequence of $\{\ve_i\}$ (denote it again by  $\{\ve_i\}$) such that
$$
\|\vi_{\ve_i}(r;c,b)-\Phi(r)\|_{C^m \left([\delta,R]\right)}\to 0 {\ \ as\ \ } i\to\infty.
$$

Using Cantor's diagonal argument we conclude that there exists a subsequence of $\{\ve_i\}$ (denote it again by $\{\ve_i\}$) such that
\begin{equation}\label{conv-to-Phi}
\vi_{\ve_i}(r;c,b) -\Phi(r) \to 0 {\ \ as\ \ } i\to\infty,\quad \mbox{ in the topology of } C_{loc}([0,R_\infty)) \mbox{ and } C^m_{loc}((0,R_\infty)),
\end{equation}
for $m\geq 1$. Hence $\Phi\in C\left([0,R_\infty)\right) \cap C^\infty ((0,R_\infty))$. It solves the equation in \eqref{TW1} and $\Phi(0)=0$.
Moreover, by Lemma \ref{lem2.1} we have
$$
\Phi'(r)\geq 0,\ \ \Phi''(r)\geq 0, \quad  r\in (0,R_\infty).
$$
Arguing as in the proof of Claim 1 in Lemma \ref{lem2.1} we even have
\begin{equation}\label{Psi>0}
\Phi'(r)>0,\ \ \Phi''(r)>0, \quad  r\in (0,R_\infty).
\end{equation}

We remark that the interval $[0,R_\infty)$, though it is obtained by taking limit for $[0,R^\ve_\infty)$, is the maximal existence interval of $\Phi$. In fact, if $\Phi$ is well-defined in $[0, R_\infty + 2\delta]$ for some $\delta>0$, then by the continuously dependence of the solution $\psi_\ve$ of \eqref{TW1-2} on its coefficient $\frac{N-1}{r+\ve}$, $\psi_\ve$ can be extended to  $[0,R_\infty +\delta]$, and so $R^\ve_\infty \geq R_\infty +\delta$ for all small $\varepsilon>0$, contradicts the definition of $R_\infty$.

For simplicity, in what follows we always denote
$$
\Psi(r):= \Phi'(r).
$$

\medskip
\noindent
{\bf Step 2. To show $\Phi'(0)=0$.} For any $0<\delta<r<R<R_\infty$, by \eqref{upper-est-psi} we have
$$
\int^r_\delta \psi_\ve (r)dr < \vi_\ve (r) = \vi_\ve (r)-\vi_\ve (0)= \int^r_\delta \psi_\ve (r)dr +\int^\delta_0\psi_\ve (r)dr \leq  \int^r_\delta \psi_\ve (r)dr + \psi_{\ve_0}(R) \delta.
$$
Taking $\ve=\ve_i$ and using the limit in \eqref{conv-to-Phi} we have
\begin{equation}\label{<Phi<}
\int^r_\delta \Psi (r)dr \leq \Phi(r) \leq \int^r_\delta \Psi(r)dr + \psi_{\ve_0}(R) \delta.
\end{equation}
Note that $\Phi'(0)$ is not well-defined till now, but
$$
0< \Psi (r)\leq \psi_{\ve_0}(R),\quad 0<r<R
$$
by \eqref{upper-est-psi}. Taking limits as $\delta\to 0$ in \eqref{<Phi<} we have
$$
\Phi(r) =\int^r_0 \Psi (r)dr \leq \psi_{\ve_0}(R) r,\quad 0<r\leq R.
$$
This implies that
$$
0\leq \limsup\limits_{r\to 0} \frac{\Phi(r)}{r} \leq \psi_{\ve_0} (R).
$$
Since $\psi_{\ve_0}(R)\to 0$ as $R\to 0$ we actually obtain $\Phi'(0) = 0$.
Therefore, $\Phi\in C^1\left([0,R_\infty)\right)$, and so it is a solution of \eqref{TW1}.

\medskip
\noindent
{\bf Step 3. To show $\Phi\in C^\infty ([0, R_\infty))$.} Using the equation \eqref{eq0}, we see that $\tilde{u}(x) :=\Phi(|x|)$ is a solution of
\begin{equation}
c=\left[\left(\delta_{ij}-\frac{D_iuD_ju}{1+|Du|^2}\right)D_{ij} u\right]^\alpha
\left(1+|Du|^2\right)^{\frac{1-\alpha}{2}}+b\sqrt{1+|Du|^2},\quad 0<|x|<R_\infty,
\end{equation}
or, equivalently,
$$
\left(\delta_{ij}-\frac{D_iuD_ju}{1+|Du|^2}\right)D_{ij} u =
\left(c-b\sqrt{1+|Du|^2} \right)^{1/\alpha} \left(1+|Du|^2\right)^{(\alpha-1)/\alpha},\quad 0<|x|<R_\infty.
$$
For any $R\in (0,R_\infty)$, the boundedness of $\Phi'$ in $[0,R]$ implies that $D\tilde{u}$ is bounded in $B_R(0)$. Then using the standard regularity theory for elliptic equations we have $\tilde{u}\in C^\infty \left(\overline{B_R(0)} \right)$. This implies that $\Phi(r)\in C^\infty\left([0,R_\infty)\right)$.

\medskip
\noindent
{\bf Step 4. Uniqueness of $\Phi$.} The uniqueness of $\Phi$ is equivalent to the uniqueness of the solution of the following problem
\begin{equation}\label{psi-TW}
\left\{
\begin{array}{l}
\displaystyle \frac{\psi'}{1+\psi^2}+\frac{N-1}{r}\psi=
\left(\frac{c}{\sqrt{1+\psi^2}}-b\right)^{{1/\alpha}}\sqrt{1+\psi^2}, \quad \ r>0,\\
 \psi(0)=0.
\end{array}
\right.
\end{equation}
Assume by contradiction that \eqref{psi-TW} has two different solutions $\psi_1(r)$ and $\psi_2(r)$, defined in $[0,R_1)$ and $[0,R_2)$, respectively.
By the uniqueness for the solution of initial value problem for the first order ODE, $\psi_1(r)\neq\psi_2(r)$ for all $0<r<\tilde{R}:= \min\{R_1,R_2\}$. Assume, without lose of generality, $\psi_2(r) > \psi_1(r)$ in $[0,\tilde{R})$.
In a similar way as proving Claim 1 in Lemma \ref{lem2.1} one can see that
$\psi'_i(r)>0\ (i=1,2)$. Hence, for any given $R\in (0,\tilde{R})$, there exists
$M_1 = M_1(\psi_1,\psi_2,R)$ such that
\begin{equation}\label{bound-psi-R}
0\leq \psi_i (r) \leq M_1,\quad r\in [0,R],\ i=1,2.
\end{equation}
Set $\eta(r) := \psi_2(r)-\psi_1(r)>0$, then $\eta$ solves
$$
\eta'=p_1(\psi_1,\psi_2)\eta+\frac{N-1}{r}\left[\psi_1(1+\psi^2_1)-\psi_2(1+\psi^2_2)\right] \leq p_1(\psi_1,\psi_2) \eta, \quad 0<r\leq R,
$$
where $p_1(\psi_1,\psi_2)\leq P$ when $r\in [0,R]$, for some $P>0$. Integrating the inequality over $[0,r]$ for any $r\in [\delta,R]$, and then taking limit as $\delta\to 0$ by using the fact $\eta(0)=0$ we obtain $\eta\equiv 0$ in $[0,R]$, a contradiction.

\medskip
\noindent
{\bf Step 5. The lower and upper bounds of $R_\infty=R_\infty (c,b)$.} Denote
$$
\zeta :=\frac{\Psi}{\sqrt{1+\Psi^2}}\in [0,1),
$$
then
\begin{equation}\label{eq-zeta}
(\zeta r^{N-1})' 
= \left(c \sqrt{1-\zeta^2} -b\right)^{{1/\alpha}} r^{N-1}.
\end{equation}
Integrating it over $[0,r]$ we have
\begin{equation}\label{rough-bounds-R-infty}
(c-b)^{1/\alpha}\frac{r}{N}\geq \zeta \geq (-b)^{1/\alpha}\frac{r}{N}.
\end{equation}
Since $\zeta(r)\to 1$ as $r\to R_\infty(c,b)$, we conclude that $R_\infty(c,b)$ has lower and upper bounds as in \eqref{R-infty-bounds}.

\medskip
\noindent
{\bf Step 6. To show $\Phi'(r),\ \Phi''(r)\to \infty$ as $r\to R_\infty(c,b)-0$.}
Since $\Phi''(r)>0$, if $\Phi'(r)\not\to \infty$, then $\Phi' (r)\to M_0\ (r\to R_\infty-0)$ for some $M_0 <\infty$. So $(r,\Psi(r))\to (R_\infty, M_0)$ as $r\to R_\infty-0$. This implies that $\Psi$ can be extended a little bit beyond $[0,R_\infty]$, so does $\Phi$, contradicts the definition of $R_\infty$.

We now prove $\Psi'(r)=\Phi''(r)\to \infty\ (r\to R_\infty-0)$. Assume by contradiction that there exist $M_1>0$ and a sequence $\{r'_n\}$ increasing to $R_\infty$ such that
$$
\Psi'(r'_n)\leq M_1.
$$
Since $\Psi (r)\to \infty$ as $r\to R_\infty -0$, $\Psi'(r) $ is unbounded in $[\frac12 R_\infty, R_\infty)$. Therefore, there exists a sequence $\{r_n\}$ increasing to $R_\infty$ such that
$$
\Psi (r_n)\to \infty\ (n\to \infty),\quad \Psi'(r_n)\leq M_1,\quad \Psi''(r_n)=0.
$$
Differentiating the equation \eqref{TW1} twice we have
\begin{eqnarray*}
 \frac{\Psi''} {1+ \Psi^2} & = &
  \frac{2 \Psi (\Psi')^2}{\left(1+ \Psi^2\right)^2} + \frac{1}{\alpha}\left(\frac{c}{\sqrt{1+ \Psi^2}}-b\right)^{\frac{1-\alpha}{\alpha}}
  \left( \frac{c(\alpha-1)\Psi \Psi'}{1+ \Psi^2} - \frac{b\alpha \Psi \Psi'}{\sqrt{1+ \Psi^2}} \right) \\
& & -(N-1)\frac{\Psi' r- \Psi}{r^2}.
\end{eqnarray*}
Taking $r=r_n$ and letting $n\to \infty$ we see that the left hand side of the equation tends to $0$ and right hand side tends to $\infty$, a contradiction.

\medskip
\noindent
{\bf Step 7. To show $\Psi(r;c,b):= \Phi'(r;c,b)$ is strictly increasing in $c$ and $-b$.} By the equation \eqref{psi-TW}, it is easily seen that $N \Psi(0)=(c-b)^{{1/\alpha}}$. Hence $\Psi(0;c,b)$ is strictly increasing in $c$ and $-b$. By contradiction it is easy to show that $\Psi(r;c,b)$ is so for any $r\in [0,R_\infty)$.

\medskip
\noindent
{\bf Step 8. To show $R_\infty = R_\infty(c,b)$ is strictly decreasing in $c$ and $-b$.} By the previous step we see that $R_\infty = R_\infty (c,b)$ is decreasing in $c$ and $-b$.
We now show that $R_\infty(c,b)$ is strictly decreasing in $c$, the proof for its strict monotonicity in $-b$ is similar.

Assume by contradiction that $b< 0$ and $c_2 > c_1>0$ but $R_0:= R_\infty(c_1,b)=R_\infty(c_2,b)$. Denote by $\mathcal{G}_i$ the graph of $\Phi(r;c_i,b)\ (i=1,2)$ in $[0,R_0]$. Then
$\mathcal{G}_2$ lies above $\mathcal{G}_1$ and they have exactly one contact point $(0,0)$ in the band $\{0\leq x<R_0\}$. We now move $\mathcal{G}_1$ leftward with a distance $d>0$ such that this shifted graph, denoted by $\mathcal{G}_1^d$, lies on the left of $\mathcal{G}_2$ and contacts $\mathcal{G}_2$ at some points. Assume  $\tilde{P}=(\tilde{r},\Phi(\tilde{r};c_2,b))$ is one of them. Then at $\tilde{P}$ we have
$$
\Phi'(\tilde{r}+d;c_1,b)=\Phi'(\tilde{r};c_2,b), \quad
H|_{\tilde{P}\in \mathcal{G}_1^d} \geq H|_{\tilde{P}\in \mathcal{G}_2},
$$
where $H|_{P\in \mathcal{G}}$ denotes the mean curvature of the graph $\mathcal{G}$ at point $P$. Substituting them into the first equation in \eqref{eq0} we have
$$
\left( H|_{\tilde{P}\in \mathcal{G}_1^d} \right)^\alpha + b =
\frac{c_1} {\sqrt{1+\left[\Phi'(\tilde{r}+d;c_1,b)\right]^2}}  <  \frac{c_2}{\sqrt{1+\left[\Phi'(\tilde{r};c_2,b)\right]^2}}
 =  \left( H|_{\tilde{P}\in \mathcal{G}_2} \right)^\alpha +b,
$$
a contradiction.

\medskip
\noindent
{\bf Step 9. To show $\Phi(R_\infty(c,b)-0;c,b)<\infty$ and it is decreasing in $c$ and $-b$.} We first prove $\Phi(R_\infty(c,b)-0;c,b)<\infty$. Taking limit as $r\to R_\infty -0$ in the equation \eqref{eq-zeta} we have
\begin{equation}\label{zeta-slope}
\zeta'(r)\to 2a := -\frac{N-1}{R_\infty} + (-b)^{1/\alpha}.
\end{equation}
By $\Psi'(r)>0$ we have $a\geq 0$. The case $a=0$ implies that $R_\infty= (N-1)(-b)^{-1/\alpha}$ is independent of $c$ which contradicts the conclusion in Step 8, so we have $a>0$.

Integrating \eqref{eq-zeta} over $[r,R_\infty)$ we have
\begin{eqnarray}
\zeta(r)& = & \frac{1}{r^{N-1}}\left[R^{N-1}_\infty-\int^{R_\infty}_r
\left( c \sqrt{1-\zeta^2(s)} -b\right)^{{1/\alpha}}s^{N-1}ds\right]\\
& \leq & \frac{R_\infty^{N-1}}{r^{N-1}} - \frac{(-b)^{1/\alpha}(R_\infty^N -r^N)}{N r^{N-1}} . \nonumber
\end{eqnarray}
As $\rho:=  R_\infty -r\to 0$ we have
$$
\frac{R_\infty^{N-1}}{r^{N-1}} = \frac{R_\infty^{N-1}}{(R_\infty -\rho)^{N-1}} = 1 + (N-1)\frac{\rho}{R_\infty} + \frac{N(N-1)}{2R_\infty^2 } \rho^2 + o(\rho^2),
$$
and so
\begin{equation}\label{zeta<rho-rho2}
\zeta(r) \leq  1- 2a \rho + \left(\frac{N-1}{2R_\infty^2} - \frac{(N-1)a}{R_\infty}\right) \rho^2 + o( \rho^2).
\end{equation}
This implies that, for some $\varepsilon>0$ sufficiently small, we have
\begin{equation}\label{est-zeta-R}
\zeta(r) \leq  1- a \rho,\quad R_\infty -\varepsilon \leq r <R_\infty,
\end{equation}
and so,
$$
\Psi(r) \leq \frac{1 - a \rho }{\sqrt{a \rho }}, \quad R_\infty -\varepsilon \leq r< R_\infty.
$$
Integrating it over $[R_\infty -\varepsilon, R_\infty)$ we have
$$
\Phi(R_\infty -0) - \Phi(R_\infty -\varepsilon) \leq \int_{R_\infty -\varepsilon}^{R_\infty} \Psi(r) dr \leq \int_0^\varepsilon \frac{1-a\rho}{\sqrt{a\rho}} d\rho <\infty.
$$
This proves $\Phi(R_\infty -0)<\infty$.

Next we prove $K(c,b) := \Phi(R_\infty(c,b)-0;c,b)$ is decreasing in $c$ and $-b$. Since the proof is similar we only consider the monotonicity in $c$. By contradiction we assume $c_2 >c_1\geq 0$ but
$K(c_1,b) < K(c_2,b)$. As in Step 8, we can move the graph $\mathcal{G}_1$ of $\Phi(r;c_1,b)$ (which lies below the graph $\mathcal{G}_2$ of $\Phi(r;c_2,b)$) leftward with a distance $d>0$ such that, the shifted graph, denoted by $\mathcal{G}_1^d$, just lies above and touches $\mathcal{G}_2$ at some points. By our assumption $K(c_1,b)<K(c_2,b)$ and the facts $\Phi'(R_\infty(c_1,b)-0)=+\infty$ we know that the two graphes do not touch at the right end point of $\mathcal{G}^d_1$: $(R_\infty(c_1,b)-d, K(c_1,b))$. Hence, a similar argument as in Step 8 leads to a contradiction.

\medskip
\noindent
{\bf Step 10. Continuous dependence of $R_\infty(c,b)$ on $c$ and on $b$.}
We only prove the continuity in $c$, the proof for $b$ is similar and simpler. Assume by contradiction that there exists an increasing sequence $\{c_m\}$ and $\bar{c}$ satisfies $c_m \nearrow \bar{c}$ (the case for $c_m \searrow \bar{c}$ is proved similar), but
$$
\lim\limits_{m\to \infty} R_\infty (c_m,b) = R_\infty (\bar{c},b) +\sigma,
$$
for some $\sigma>0$. Denote
$$
\bar{R}:= R_\infty(\bar{c},b),\quad R_m := R_\infty(c_m,b),
$$
and, without loss of generality, we assume $\bar{R} +\sigma < R_m < \bar{R}+2\sigma$ for all $m$. Denote
$$
\bar{\zeta}(r) := \frac{\Psi(r;\bar{c},b)}{\sqrt{1+ \Psi^2(r;\bar{c},b)}},
\quad
\zeta_m (r) := \frac{\Psi(r;c_m,b)}{\sqrt{1+ \Psi^2(r;c_m,b)}},
$$
and
$$
\bar{\zeta}'(\bar{R}-0) = 2\bar{a}:= (-b)^{1/\alpha} -\frac{N-1}{\bar{R}},\quad \zeta'_m (R_m -0) = 2 a_m := (-b)^{1/\alpha}-\frac{N-1}{R_m},
$$
then
\begin{equation}\label{am>bara}
2a_m > 2\bar{a}+ A_1 ,\quad \mbox{with } A_1:= \frac{(N-1)\sigma}{\bar{R}(\bar{R}+\sigma)}.
\end{equation}
For any given $\sigma_0>0$ satisfying
\begin{equation}\label{def-sigma-0}
\sigma_0 (-b)^{1/\alpha} <A_1,
\end{equation}
there exists $\varepsilon_1 >0$ such that
\begin{equation}\label{nonlinear-est}
\left(\bar{c}\sqrt{1-\bar{\zeta}^2 (r)} -b \right)^{1/\alpha} \leq (1+\sigma_0)(-b)^{1/\alpha},\quad r\in J:= [\bar{R}-\varepsilon_1, \bar{R}).
\end{equation}
Integrating \eqref{eq-zeta} with $c=\bar{c}$ over $[r,\bar{R})$ and using a similar argument as proving \eqref{zeta<rho-rho2} we have
\begin{eqnarray*}
\bar{\zeta}(r) & \geq & \frac{\bar{R}^{N-1}}{r^{N-1}} - \frac{(1+\sigma_0) (-b)^{1/\alpha} ( \bar{R}^N - r^N)}{N r^{N-1}}\\
& = & 1- A_2 (\bar{R}-r) + A_3 (\bar{R}-r)^2 + o((\bar{R}-r)^2)\mbox{\ \ as\ \ }r\to \bar{R}-0,
\end{eqnarray*}
where
$$
A_2 := 2\bar{a} +\sigma_0 (-b)^{1/\alpha},\quad
A_3 := \frac{N-1}{2\bar{R}^2} -\frac{(N-1)\bar{a}}{\bar{R}} -\sigma_0 (-b)^{1/\alpha} \frac{N-1}{2\bar{R}}.
$$

On the other hand, in a similar way as proving \eqref{zeta<rho-rho2}
we have
$$
\zeta_m (r) \leq  1- 2 a_m (R_m -r) + A_4  (R_m -r)^2 + o((R_m -r)^2)\mbox{\ \ as\ \ }r\to R_m-0,
$$
where
$$
A_4 := \frac{N-1}{2R_m^2} - \frac{(N-1)a_m}{R_m}.
$$
Therefore, when $\varepsilon_1>0$ is sufficiently small, we have
\begin{eqnarray*}
\zeta_m (\bar{R}-\varepsilon_1) & \leq & \zeta_m (R_m - 2 \varepsilon_1) \leq 1- 4 a_m  \varepsilon_1 +  8 A_4  \varepsilon^2_1\\
& \leq & 1- 2 (2\bar{a} +A_1) \varepsilon_1 + \bar{a} \varepsilon_1\\
& \leq & 1-  A_2 \varepsilon_1 - \bar{a} \varepsilon_1 - A_1 \varepsilon_1\\
& \leq & \bar{\zeta}(\bar{R}-\varepsilon_1)  - A_1 \varepsilon_1.
\end{eqnarray*}
This contradicts the fact $\zeta_m(r)\to \bar{\zeta}(r)$ uniformly on the interval $[0,\bar{R}-\varepsilon_1]$. Note that in this interval all of the terms in the equation \eqref{eq-zeta}  as well as their derivatives are regular due to $\zeta_m (r) \leq \bar{\zeta}(\bar{R}-\varepsilon_1)<1$. This proves the continuous dependence of $R_\infty(c,b)$ on $c$.

\medskip
\noindent
{\bf Step 11. To show the limits of $R_\infty(c,b)$ in \eqref{limits-R-infty}.}  The first limit as $c\to 0$ follows from \eqref{R-infty-bounds} directly. The second limits follows from the following Claim:

\medskip
\noindent
{\bf Claim.} For any $b<0$ and any $R\in (0, N(-b)^{-1/\alpha})$, there exists a unique $c=c(b,R)$ such that $R_\infty (c(b,R),b)=R$.

\medskip
\noindent
Without loss of generality, we assume $N (-b)^{-1/\alpha} >1$,
and prove the claim for $R=1$. If $R_\infty (c', b)\leq 1$ for some $c'>0$, then by Steps 8 and 10, there exists a unique $c=c(b)\in (0,c']$ such that $R_\infty(c(b),b)=1$.
Hence, we only need to show that $R_\infty (c',b)\leq 1$ for some large $c'>0$.
Assume by contradiction that this is not true. Then there exists an increasing sequence $\{c_m\}$ with $c_m \nearrow \infty\ (m\to \infty)$, but $R_\infty(c_m,b)\searrow \tilde{R}\geq 1$. Denote the corresponding solutions of \eqref{eq-zeta}, \eqref{psi-TW} and \eqref{TW1} with $c=c_m$ by $\zeta_m$, $\Psi_m$ and $\Phi_m$, respectively. Since $\zeta_m$ is increasing in $m$ we have
$$
\zeta_m(r) \nearrow \tilde{\zeta}(r)\mbox{\ \ as\ \ }m\to \infty.
$$
We show that $\tilde{\zeta}(r)\equiv 1$. For, otherwise, $h:= \tilde{\zeta}(2r_0)<1$ for some $r_0 \in (0,\frac12)$, then
$$
\zeta_m (r)\leq \zeta_m (2r_0)\leq h:= \tilde{\zeta}(2r_0) <1,\quad 0\leq r\leq 2 r_0,
$$
and so, on $[r_0,2r_0]$, we have
$$
\zeta'_m(r)  =  \left( c_m \sqrt{1-\zeta_m^2} -b\right)^{1/\alpha} -\frac{N-1}{r}\zeta_m  \geq  \left( c_m \sqrt{1-h^2} -b\right)^{1/\alpha} -\frac{N-1}{r_0}  > \frac{1}{r_0},
$$
provided $m$ is sufficiently large. Integrating it over $[r_0,2r_0]$ we obtain
$$
\zeta_m (2r_0)> \zeta_m(2r_0)- \zeta_m(r_0) = \int_{r_0}^{2r_0} \zeta'_m(r) dr >1,
$$
a contradiction. Therefore, for any $r_1\in (0,1]$ we have
$$
\zeta_m (r)\geq \zeta_m(r_1)\to \tilde{\zeta}(r_1) =1 \mbox{\ \ as\ \ }m\to \infty,\quad r\in [r_1,1].
$$
This means that $\zeta_m(r)\to 1$ as $m\to \infty$ in the topology of $L^\infty_{loc} ((0,1])$.

As in Step 9, we denote
$$
K(c_m, b):= \Phi_m (R_\infty(c_m,b)-0;c_m,b),\quad m=1,2,\cdots.
$$
Choose
$$
0< \varepsilon \leq \frac{1}{2[1+4(K(c_1,b))^2 ]},
$$
then there exists $m_0>0$ such that when $m\geq m_0$ we have
$$
1-\varepsilon \leq \zeta_m(r)<1,\quad \frac12 \leq r\leq 1.
$$
This implies that, when $\varepsilon>0$ is small enough, we have
$$
\Psi_m(r) \geq 2K(c_1,b),\quad \frac12 \leq r\leq 1.
$$
Integrating it over $[\frac12,1]$ we have
$$
K(c_m,b):= \Phi_m (R_\infty(c_m,b)-0) \geq \Phi_m (1)
> \int_{\frac12}^1 \Psi_m (r)dr \geq K(c_1,b),
$$
contradicts the conclusion in Step 9. This proves the claim.

Summarizing the above steps we finish the proof for Theorem \ref{thm:TSb<0} (II).
\hfill $\Box$

\medskip
\noindent
{\it Proof of Theorem \ref{thm:TSb<0} {\rm (I)}.} We now consider the case $b=0$. The construction of the solution $\Phi$, as well as its uniqueness and the properties
$$
\Phi'(0)=0,\quad \Phi\in C^\infty ([0,\infty)),\quad \Phi'(r),\ \Phi''(r)>0 \mbox{ for } r>0
$$
are proved in the same way as in Steps 1-4 in the above proof for (II). In what follows we prove the asymptotic profile of $\Phi(r)$ in \eqref{phi-shape-b=0} as $r\to \infty$.

\medskip
\noindent
{\bf Step 1. To show
\begin{equation}\label{lower<psi<upper}
A_1 r^\alpha \leq \Psi(r) \leq A_2 r^\alpha,\quad r\gg 1,
\end{equation}
for some $A_2>A_1>0$ depending only on $N,\alpha$ and $c$.
}

We take $R>0$ sufficiently large such that
\begin{equation}\label{choose-large-R}
c^2 R^{2\alpha} \geq 2 (2N)^{2\alpha},\quad \Psi(R)\geq 1.
\end{equation}
Then for any $r_2 > r_1 > R$ and any $r\in [r_1,r_2]$ we have
$$
P(r_2) \leq P(r)\leq P(r_1),\quad \mbox{with } P(r) := \left( \frac{c}{\sqrt{1+\Psi^2(r)}} \right)^{{1/\alpha}}.
$$
Integrating the equation \eqref{eq-zeta} over $[r_1,r_2]$ we obtain
\begin{equation}\label{<zeta-r2-r1<}
\frac{P(r_2)}{N}(r^N_2 -r^N_1) \leq \zeta(r_2) r_2^{N-1} - \zeta(r_1) r_1^{N-1}
\leq \frac{P(r_1)}{N}(r^N_2 -r^N_1).
\end{equation}
By the first inequality we have
$$
P(r_2) (r_2^N -r_1^N) \leq N r_2^{N-1}.
$$
Taking $r_2^N = 2r_1^N$ and using the formula of $P(r)$ we obtain
$$
1+\Psi^2(r_2) \geq \frac{c^2}{(2N)^{2\alpha}} r_2^{2\alpha}.
$$
Furthermore, by the choice of $R$, when $r_2 \gg R$ we have
\begin{equation}\label{Psi-lower-est}
\Psi(r_2) \geq A_1 r_2^\alpha,\quad \mbox{ with } A_1 := \frac{c}{\sqrt{2}(2N)^\alpha}.
\end{equation}
In a similar way, by the second inequality in \eqref{<zeta-r2-r1<} and the choice of $R$ we have
$$
P(r_1) \frac{r_2^N -r_1^N}{Nr_2^{N-1}}  \geq  \zeta(r_2) - \frac{r_1^{N-1}}{r_2^{N-1}}\geq \frac12 - \frac{r_1^{N-1}}{r_2^{N-1}}.
$$
Taking $r_1>R$ and $r_2 = k r_1 $ with $k^{N-1}=4$, then we have
$$
\left(\frac{c}{\sqrt{1+\Psi^2(r_1)}}\right)^{{1/\alpha}} \geq \frac{N}{4 r_1}\cdot \frac{k^{N-1}}{k^N -1} \geq \frac{N}{4^{\frac{N}{N-1}} r_1},
$$
and so
\begin{equation}\label{Psi-upper-est}
\Psi (r_1) \leq A_2 r_1^\alpha,\quad \mbox{ with } A_2:= \frac{4^{\frac{N\alpha}{N-1}}c}{N^{\alpha}} .
\end{equation}
Combining with \eqref{Psi-lower-est} we obtain \eqref{lower<psi<upper}.

\medskip
\noindent
{\bf Step 2. To show \eqref{phi-shape-b=0}.} Set
$$
\xi(r) := \Psi(r) r^{-\alpha},\quad r\in J_1 := [1,\infty).
$$
Then by \eqref{lower<psi<upper} and $\Psi'(r)>0$ we have
\begin{equation}\label{xi-property}
A_1^0 <\xi(r) \leq A_2^0,\quad \xi'(r) +\alpha \frac{\xi(r)}{r}>0,\quad r\in J_1,
\end{equation}
for some $A_2^0>A_1^0>0$. Substituting $\Psi=\xi (r) r^\alpha$ and $b=0$ into \eqref{psi-TW} we have
$$
\frac{[\xi' r^\alpha +\alpha r^{\alpha-1}\xi](1+o(1))}{\xi^2 r^{2\alpha}} +
(N-1)\xi r^{\alpha-1} = (1+o(1)) c^{{1/\alpha}} \xi^{1-{1/\alpha}} r^{\alpha-1} \mbox{\ \ as\ \ }r\to \infty.
$$
Here we used the expression $\sqrt{1+\xi^2 r^{2\alpha} } = (1+o(1))\xi r^\alpha\ (r\to \infty)$. Multiplying the equality by $\xi^2 r^{1-\alpha}$ we have
\begin{equation}\label{eq-xi}
[\xi' r^{1-2\alpha} +\alpha r^{-2\alpha}\xi ](1 +o(1)) = F_1(\xi;c,\alpha,N) (1+o(1)) \xi^3,\quad r\to \infty,
\end{equation}
where
\begin{equation}\label{def-c-xi}
F_1 (\xi; c,\alpha,N) := c^{{1/\alpha}}\xi^{-{1/\alpha}} - (N-1).
\end{equation}
By \eqref{xi-property} and \eqref{eq-xi} we see that $F_1 >0$. We then consider three cases.

(1). In case $F_1 (\xi(r);c,\alpha,N)\searrow 0$ as $r\to \infty$, we obtain
\begin{equation}\label{limit-xi}
\xi(r)\to \bar{\xi}:= \frac{c}{(N-1)^\alpha}\mbox{\ \ as\ \ }r\to \infty.
\end{equation}
This is what we desired.

(2). In case $F_1 \geq 2\delta_1 >0$ for some $\delta_1>0$ and all large $r$, then \eqref{xi-property} and \eqref{eq-xi} yield
$$
\xi' r^{1-2\alpha}\geq \delta_1 \xi^3\geq \delta_1 (A_1^0)^3,\quad r\gg 1.
$$
By integrating we obtain $\xi(r)\to \infty$ as $r\to \infty$, contradicts \eqref{xi-property}.

(3). The left case is that, for some $\delta_2>0$ and two sequences $\{r'_n\},\ \{s'_n\}$, there hold
$$
F_1 (\xi(r'_n);c,\alpha,N)\geq 2 \delta_2,\quad F_1 (\xi(s'_n);c,\alpha,N)\to 0\ \ (n\to \infty).
$$
Therefore, we can find a sequence $\{r_n\}$ such that each $r_n$ is a local minimum of $\xi(r)$ and
$$
F_1 (\xi(r_n);c,\alpha,N)\geq \delta_2.
$$
Taking $r=r_n$ in the equation \eqref{eq-xi} and taking limit as $n\to \infty$ we obtain $\xi(r_n)\to 0$ as $n\to \infty$, which contradicts \eqref{xi-property} again.

Therefore the limit \eqref{limit-xi} must hold and so we obtain
\begin{equation}\label{psi-shape-b=0}
\Psi(r) = \left[ \frac{c}{(N-1)^\alpha} +o(1)\right]r^\alpha,\quad r\to \infty.
\end{equation}
Integrating it over $[0,r]$ and letting $r\to \infty$ we obtain \eqref{phi-shape-b=0}.

\medskip

We finally prove the limit \eqref{limits-Phi'-to-infty}. For any given $r_0>0$, the limit $\Phi'(r_0;c,0)\to \infty\ (c\to \infty)$ is equivalent to say that, for any given small $\varepsilon>0$, there exists $C>0$ large such that when $c\geq C$, there holds
$$
\zeta(r_0;c,0) :=\frac{\Psi(r_0;c,0)}{\sqrt{1+\Psi^2(r_0;c,0)}} \geq 1-\varepsilon.
$$
By contradiction, we assume there exists $c_m \nearrow \infty$ such that $\zeta(r_0;c_m,0)\leq 1-\varepsilon$. Then for $r\in [0,r_0]$ we have $\zeta(r;c_m,0)\leq \zeta(r_0;c_m,0)\leq 1-\varepsilon$. Integrating \eqref{eq-zeta} over $[0,r_0]$ we have
\begin{eqnarray*}
\zeta(r_0,c_m,0)r_0^{N-1} & = & \int_0^{r_0} \left(c_m \sqrt{1-\zeta^2(r;c_m,0)}\right)^{1/\alpha} r^{N-1} dr\\
& \geq  & \int_0^{r_0} \left(c_m {\varepsilon}\right)^{1/\alpha} r^{N-1} dr
= (c_m \varepsilon)^{1/\alpha} \frac{r_0^N}{N},
\end{eqnarray*}
which contradicts $\zeta(r_0;c_m,0)\leq 1-\varepsilon$ and $c_m\to \infty\ (m\to \infty)$. Therefore, the limit in \eqref{limits-Phi'-to-infty} holds.

This proves Theorem \ref{thm:TS-in-whole} (I).
\hfill $\Box$

\subsection{TSs in the cylinder}
We now select some TSs from above so that they satisfy the boundary condition \eqref{TS-BC} with $k>0$.

\medskip
\noindent
{\it Proof of Theorem \ref{thm:TS-in-cylinder} in case {\rm (a)} or {\rm (b)} holds.} Let $\Phi(r;c,b)$ be the solution of \eqref{TW1} and $\Psi:=\Phi'$, $\zeta:= \frac{\Psi}{\sqrt{1+\Psi^2}}$ as above. For any $r\in (0,R_\infty)$ we have
$$
 1\geq  \sqrt{1-\zeta^2 (t)} \geq  \sqrt{1-\zeta^2(r)} ,\quad t\in [0,r].
$$
Using this inequality and integrating the equation \eqref{eq-zeta} over $[0,r]$ we have
\begin{equation}\label{>zeta>}
(c-b)^{{1/\alpha}} \frac{r}{N} \geq \zeta(r)  \geq \left( c \sqrt{1-\zeta^2(r)} -b \right)^{{1/\alpha}} \frac{r}{N},\quad 0\leq r<R_\infty.
\end{equation}

\medskip
{\bf Case (a): $b=0<k$.} When $c=0$, it follows from \eqref{>zeta>} that $\zeta(r)\equiv \Psi(r)\equiv 0$, and so $\Psi(1;0,0)=0$.
On the other hand, when $c=c_1:= N^\alpha \sqrt{1+k^2}$, if we assume $\Psi(1;c_1,0)\leq k$, then the second inequality in \eqref{>zeta>} implies that $\zeta(1)\geq 1$. This contradicts the fact $\zeta (r)<1$. Therefore, $\Psi(1;c_1,0)>k$.
Since $\Psi(1;c,b)$ is strictly increasing in $c$, there exists a unique $c=c(k)\in (0,c_1)$ such that the solution $\Psi(r;c(k),0)$ of \eqref{psi-TW} with $b=0,\ c=c(k)$ satisfies $\Psi(1)=k$.

\medskip
{\bf Case (b): $b<0<k$ and the following assumption holds:
\begin{equation}\label{negative-b}
(-b)^{{1/\alpha}} < \frac{k N}{\sqrt{1+k^2}}.
\end{equation}
}
When $c=0$, by \eqref{>zeta>} we have
$$
\zeta(r;0,b) =\frac{(-b)^{1/\alpha}}{N} r,\quad 0\leq r<R_\infty(0,b).
$$
Hence $R_\infty(0,b) = N(-b)^{-1/\alpha}$. The condition \eqref{negative-b} then implies that
$$
R_\infty(0,b) >\frac{\sqrt{1+k^2}}{k}\geq 1
$$
and
$$
\zeta(1;0,b) =\frac{(-b)^{1/\alpha}}{N} < \frac{k}{\sqrt{1+k^2}} \ \ \Leftrightarrow\ \ \Psi(1;0,b)<k.
$$
On the other hand, by the Claim in Step 11 in the proof of Theorem
\ref{thm:TSb<0} {\rm (II)},  there exists $c(b)>0$ such that $R_\infty(c(b),b)=1$, that is, $\Psi(1;c(b),b)=\infty$. Since $\Psi(1;c,b)$ is strictly increasing in $c$, there exists a unique $\tilde{c}= \tilde{c}(b,k)\in (0,c(b))$ such that $\Phi(r;\tilde{c},b)$ satisfies the boundary condition \eqref{TS-BC}.

Finally, we consider the case (b) with $k=\infty$. In this case, the condition \eqref{negative-b} is $(-b)^{1/\alpha}<N$. By the Claim in Step 11 in the proof of Theorem \ref{thm:TSb<0} (II) , $R_\infty(c(b),b)=1$ for some $c(b)>0$, and so  $\Psi(1;c(b),b)=k=\infty$.

This proves Theorem \ref{thm:TS-in-cylinder} in case (a) or (b) holds, including the case $b<0,\ k=\infty$ and $(-b)^{1/\alpha}<N$.
\hfill $\Box$


\section{Translating Solutions in Case $b>0$}

In this section we study the TSs in case $b>0$.
In the first subsection we consider the case $c\geq b>0$, and in the second section we consider the case $b>c>0$.
We will use notations as in the previous section, that is, denote the solution of \eqref{TW1} by $\Phi(r;c,b)$, and write
$$
\Psi(r;c,b):=\Phi'(r;c,b),\quad \zeta(r;c,b):= \frac{\Psi(r)}{\sqrt{1+\Psi^2(r)}}.
$$

\subsection{TSs in case $c\geq b>0$}
We will prove the following result.

\begin{thm}\label{thm:TSc>b>0}
For any $b,c$ satisfying $c\geq b >0$, the problem \eqref{TS} has a unique solution $\varphi=\Phi(r;c,b)\in C^\infty([0,\infty))$. $\Phi(r;c,b)$ and $\Phi'(r;c,b)$ are strictly increasing in $c$ and $-b$. Moreover,

\begin{enumerate}[{\rm (I).}]
 \item when $c>b>0$, $\Phi'(r),\ \Phi''(r)>0$ in $(0,\infty)$ and
\begin{equation}\label{limit-Psi}
 \Phi'(r)\to \Psi_0:= \frac{\sqrt{c^2-b^2}}{b}\mbox{\ \  as\ \ } r\to \infty;
\end{equation}

 \item when $c=b>0$, $\Phi(r)\equiv 0$.
\end{enumerate}
\end{thm}

\begin{proof}
The proof is similar as in the previous section.
\end{proof}

\medskip
In case $c>b>0$, by the limit \eqref{limit-Psi} we see that, for any small $\varepsilon>0$, there exists $R_\varepsilon>0$ such that $|\Psi(r)-\Psi_0|\leq \varepsilon$ for $r\geq R_\varepsilon$.  Hence, there exists $M>0$ such that
\begin{equation}\label{app-V-shape}
(\Phi_0-\varepsilon) r -M \leq \Phi(r) \leq (\Psi_0 +\varepsilon)r +M,\quad r\geq 0.
\end{equation}
In other words, the shape of $\Phi$ is an {\it approximate-V-shaped} one.
One may expect to improve it to be a
{\it V-shaped} one (that is, to delete the error $\varepsilon$), as it was shown in \cite{Lou1, NT1, NT2} etc. for the problem \eqref{TW1} in case $N=\alpha =1$. We now show that this is generally impossible in higher space dimension.

\begin{prop}\label{prop:unbound-V}
Assume $N>1$, $0<\alpha\leq 1$ and $c>b>0$. Then
$$
|\Phi(r) -\Psi_0 r|\to \infty \mbox{\ \ as\ \ }r\to \infty.
$$
\end{prop}

\begin{proof}
Denote $\eta:= \Psi_0 - \Psi(r)$, then $\eta(r)>0,\ \eta(r)\to 0\ (r\to \infty)$ and $\eta'(r)=-\Psi'(r)<0$. By \eqref{psi-TW} we have
$$
\frac{-\eta'}{1+(\Psi_0-\eta)^2}  +\frac{N-1}{r}(\Psi_0 -\eta) = \left( \frac{c}{\sqrt{1+(\Psi_0-\eta)^2}} -\frac{c}{\sqrt{1+\Psi_0^2}} \right)^{1/\alpha} \sqrt{1+(\Psi_0 -\eta)^2},
$$
that is,
$$
\eta'(r) = \frac{-F(\eta)}{A(\eta)} \eta^{1/\alpha} + \frac{N-1}{A(\eta)r} (\Psi_0 -\eta),
$$
with
$$
A(\eta):= \frac{1}{1+(\Psi_0 -\eta)^2} ,
$$
and
$$
F(\eta)  :=  \left[\frac{c(2\Psi_0 -\eta)}{\sqrt{1+(\Psi_0-\eta)^2}\cdot \sqrt{1+\Psi_0^2} \cdot [\sqrt{1+(\Psi_0-\eta)^2}+\sqrt{1+\Psi_0^2}]} \right]^{1/\alpha}  \sqrt{1+(\Psi_0-\eta)^2}.
$$
Therefore, when $r\gg 1$, that is, $0<\eta\ll 1$, we have
$$
A_0:= \frac{1}{1+\Psi_0^2} < A(\eta) \leq \frac32 A_0,\quad F(\eta)\leq F^0:= \frac{(2c\Psi_0)^{1/\alpha}}{(1+\Psi_0^2)^{(3-\alpha)/(2\alpha)}}.
$$
In case $0<\alpha\leq 1$ we have
\begin{equation}\label{ineq-eta}
\eta'(r) \geq -B_1 \eta^{1/\alpha} + \frac{B_2}{r} \geq -B_1 \eta +\frac{B_2}{r},\quad r\geq r_1,
\end{equation}
for sufficiently large $r_1$, where
$$
B_1 := \frac{F^0}{A_0} >0,\quad B_2 := \frac{(N-1)\Psi_0}{2A_0}>0.
$$
Solving the inequality \eqref{ineq-eta} we obtain
\begin{equation}\label{temp-eta1}
\eta(r) \geq \eta(r_1)e^{B_1(r_1 -r)} +B_2 F_1(r), \quad r>r_1,
\end{equation}
where
$$
F_1(r):= e^{-B_1 r} \int_{r_1}^r \frac{e^{B_1 t}}{t} dt \sim \frac{1}{B_1 r} \mbox{\ \ as\ \ }r\to \infty.
$$
Thus, there exists $r_2>r_1$ such that $F_1(r)\geq \frac{1}{2B_1 r} $ for $r\geq r_2$. Integrating \eqref{temp-eta1} we have
$$
\Psi_0 r - \Phi(r) = \int_0^{r} \eta(t) dt >\int_{r_2}^r \left[ \eta(r_1)e^{B_1(r_1 -t)}  + \frac{B_2}{2B_1 t}  \right] dt \to \infty \ \ \mbox{ as }\ \ r\to \infty.
$$
This proves the proposition.
\end{proof}

From above we see that, for any $c>b$, there exists an approximate-V-shaped TS with speed $c$. We select one from them such that it satisfies the boundary condition \eqref{TS-BC} for $k\geq 0$.

\medskip
\noindent
{\it Proof of Theorem \ref{thm:TS-in-cylinder} in case {\rm (c)} holds.}
Note that we are considering the case $k>0$. Take $c_1 := (N^\alpha +b)\sqrt{1+k^2}$, we show that $\Psi(1;c_1,b)>k$. Otherwise,
$$
\Psi(r;c_1,b)\leq \Psi(1;c_1,b)\leq k,\quad 0\leq r\leq 1.
$$
Integrating the equation \eqref{eq-zeta} over $[0,1]$ we have
$$
\zeta(1) = \int_0^1 \left(\frac{c_1}{\sqrt{1+\Psi^2(r)}} -b\right)^{1/\alpha} r^{N-1} dr \geq 1,
$$
a contradiction. On the other hand, $\Psi(1;b,b)=0$. So there exists a unique $\tilde{c}=\tilde{c}(b,k)\in (b,c_1)$ such that the solution $\Psi(r;\tilde{c},b)$ of \eqref{psi-TW} satisfies $\Psi(1; \tilde{c},b)=k$.
\hfill $\Box$

\subsection{The case $b>c>0$.}\label{subsec:TSb>c>0}
In this subsection, we focus on the case $b>c>0$. In this case the TS has a concave profile, and so the mean curvature $H<0$. A suitable choice for the power $\alpha>0$ should be as the following:
\begin{center}
For $H<0$, $H^\alpha$ is well defined and $H^{\alpha-1}>0$.
\end{center}
The latter is required to ensure that the equation \eqref{E} is a parabolic one.
Therefore, it is natural to assume
\begin{equation}\label{choice-alpha}
\alpha=\frac{q}{p} \mbox{\ \ for some positive odd integers\ } p,\ q,
\end{equation}
such as $\alpha =3, \frac13$, etc..

We will prove the following main result in this subsection.

\begin{thm}\label{thm:TSb>c>0}
Assume \eqref{choice-alpha} and $b>c>0$. Then the problem \eqref{TS} has a unique solution $\varphi=\Phi(r;c,b)\in C^\infty([0,R_\infty))$ for some $R_\infty = R_\infty (c,b)<\infty$. $\Phi(r;c,b)$ and $\Phi'(r;c,b)$ are strictly increasing in $c$ and $-b$. Moreover, $R_\infty(c,b)$ satisfies
\begin{equation}\label{R-infty-bounds-b>0}
N b^{-1/\alpha} \leq R_\infty (c,b) \leq N (b-c)^{-1/\alpha},
\end{equation}
it is continuously dependent on and strictly increasing in $c$ with
$$
R_\infty(c,b)\to Nb^{-1/\alpha} \mbox{\ \ as\ \ }c\to 0,\quad R_\infty(c,b)\to \infty  \mbox{\ \ as\ \ } c\to b-0.
$$
In addition, $\Phi'(r),\ \Phi''(r)<0$ in $(0,R_\infty)$, and
\begin{equation}\label{phi-psi-b>0}
  \Phi( R_\infty-0) > -\infty,\quad  \Phi'(R_\infty-0) =   \Phi''(R_\infty-0) = -\infty,
\end{equation}
and $\Phi(R_\infty(c,b)-0;c,b)$ is increasing in $c$.
\end{thm}

\begin{proof}
The proof is parallel to that of Theorem \ref{thm:TS-in-whole} (ii) in the previous section. In particular, Steps 1-8 and Step 10 are similar (may be with opposite monotonicity). Only Step 9 requires modification.

\medskip
\noindent
{\bf Modified Step 9. To show $\Phi(R_\infty(c,b)-0;c,b)>-\infty$ and it is increasing in $c$ and $-b$.} The monotonicity in $c$ and $-b$ can be proved in a similar way as in Step 9 before. We only prove $\Phi(R_\infty(c,b)-0;c,b)>-\infty$.

Taking limit as $r\to R_\infty -0$ we have
$$
\zeta(r)\to -1 \quad \mbox{and} \quad \zeta'(r)\to  a_1:= \frac{N-1}{R_\infty(c,b)} -b^{1/\alpha}.
$$
The latter follows from the equation \eqref{eq-zeta}. By $\Psi'(r)<0$ we have $a_1\leq 0$. Since $a_1 =0$ implies that $R_\infty(c,b)$ is independent of $c$ which contradicts the conclusion in Step 8, we have $a_1 <0$, and so, for some small $\delta>0$ we have
$$
2 \tilde{a} := (b-\delta)^{1/\alpha} - \frac{N-1}{R_\infty(c,b)} >0.
$$
Integrating \eqref{eq-zeta} over $[r,R_\infty)$ we have
$$
\zeta(r) =  -\frac{R_\infty^{N-1}}{r^{N-1}} +\frac{1}{r^{N-1}} \int^{R_\infty}_r \left( b-c \sqrt{1-\zeta^2(s)}\right)^{{1/\alpha}}s^{N-1}ds
$$
As $r\to R_\infty -0$ we have $c\sqrt{1-\zeta^2(r)} < \delta$, and
$$
\frac{R_\infty^{N-1}}{r^{N-1}} = 1 + (N-1)\frac{R_\infty -r}{R_\infty} +  o(R_\infty -r).
$$
Thus, as $r\to R_\infty -0$, there holds
\begin{eqnarray*}
\zeta(r) & \geq &  -1 -\frac{N-1}{R_\infty} (R_\infty -r) + \frac{(b-\delta)^{1/\alpha}}{N r^{N-1}} (R_\infty^N -r^N) + o(R_\infty -r)\\
& \geq & -1 + \left[(b-\delta)^{1/\alpha}-\frac{N-1}{R_\infty}\right] (R_\infty -r) + o(R_\infty -r).
\end{eqnarray*}
This implies that, for some $\varepsilon>0$,
$$
\zeta(r) \geq  -1 + \tilde{a}(R_\infty -r),\quad R_\infty -\varepsilon \leq r <R_\infty,
$$
and so,
$$
\Psi(r) \geq - \frac{1}{\sqrt{\tilde{a} (R_\infty -r)}}, \quad R_\infty -\varepsilon \leq r< R_\infty.
$$
Integrating it over $[R_\infty -\varepsilon, R_\infty)$ we obtain
$\Phi(R_\infty -0;c,b)>-\infty$.

This completes the proof of the theorem.
\end{proof}

\medskip
Now we select some of them to be a TS in the cylinder satisfying the boundary condition $\varphi'(1)=k$.

\medskip
\noindent
{\it Proof of Theorem \ref{thm:TS-in-cylinder} in case {\rm (d)} holds.}
In case $k=-\infty$, the condition $b^{{1/\alpha}} \sqrt{1+k^2}> -k N$ reduces to
$R_\infty (0,b)= Nb^{-1/\alpha} <1$. Since $R_\infty (c,b)$ is continuously dependent on and strictly increasing in $c$, and $R_\infty(b,b)=\infty$, there exists a unique $c=c_1(b)$ such that $R_\infty(c_1(b),b)=1$, or, equivalently, $\Psi (1;c_1(b),b)=-\infty$.

In case $k\in (-\infty, 0)$, as in the proof of Theorem \ref{thm:TS-in-cylinder} (b) in Subsection 2.2, there exists a unique $\tilde{c}(b,k)>0$ such that $\Phi'(1;\tilde{c},b)=k$. This proves the existence of TS of \eqref{TS}-\eqref{TS-BC} for $b>0$ and $k\in [-\infty,0)$.
\hfill $\Box$



\section{Asymptotical Stability of the TSs: the Case $b\leq 0$}

In this section we study the asymptotical stability for the TSs of \eqref{E} with $b\leq 0$ in the cylinder $\Omega$ and with prescribed boundary angles:
\begin{equation}\label{BC1}
u_r(0,t)=0,\quad u_r(1,t)=k,\quad t>0.
\end{equation}
Since all the TSs in case $b\leq 0$ have convex profiles, it is natural to require the boundary slope is positive: $k>0$ and to consider the solutions of \eqref{E}-\eqref{BC1} with positive mean curvature:
\begin{equation}\label{def-H}
H(r,t) = \frac{u_{rr}}{(1+u_r^2)^{3/2}} + \frac{(N-1) u_r}{r \sqrt{1+u_r^2}}>0,\quad r\in [0,1].
\end{equation}
Note that, under this condition, $u_{rr}>0$ is not necessarily to be true, but $u_r>0$ is true (see details in the proof of Theorem \ref{thm:est-b<0}). In the first subsection we will give some a priori estimates, including the {\it uniform positive lower and upper bounds for $H$}. In the second subsection we prove the asymptotic stability of the TS, that is, any solution of the initial boundary value problem converges to the TS.

\subsection{A priori estimates}

To derive the global existence, one needs the $C^2$ estimate for the solutions of fully nonlinear parabolic equations. We will see below that, the $C^0$ and $C^1$ estimates are direct consequences of the maximum principle, as long as the equation is a parabolic one. However, the $C^2$  estimates and the parabolicity itself is not so obvious. Furthermore, considering our final aim is to show the convergence of the solution to the TS, we even need {\it uniform-in-time} estimates for $u_r, u_{rr}$, as well as the {\it uniform-in-time positive lower and upper bounds for $H$}. In other words, the estimate like $H>0$, though it is enough in studying the global existence, but it is not enough to give the convergence of the solution to TS. In this part we will give required uniform estimates by using the maximum principle.

\begin{thm}\label{thm:est-b<0}
Assume $u(r,t)$ is a classical solution, in the time interval $[0,T)$, of the problem \eqref{E}-\eqref{BC1} with initial data as in \eqref{cond-IC}. Assume also $b\leq 0<k$, $(-b)^{1/\alpha}\sqrt{1+k^2}<kN$ and
\begin{equation}\label{cond-H-b<0<k}
H^\alpha (r,0)+b>0,\quad r\in [0,1].
\end{equation}
Then there exist positive constants $M_0, M_1, M_2, V_*, V^*, H_*, H^*$, independent of $t$ and $T$, such that
\begin{eqnarray}
|u(r,t) - \tilde{c}t|\leq M_0,\quad r\in [0,1],\ t\in [0,T);\\
0< u_r(r,t) \leq M_1,\quad r\in (0,1],\ t\in [0,T);\\
V_* \leq u_t(r,t) \leq V^*,\quad r\in [0,1],\ t\in [0,T);\\
H_* \leq H(r,t) \leq H^*, \quad r\in [0,1],\ t\in [0,T);\\
|u_{rr}(r,t)|\leq M_2,\quad r\in [0,1],\ t\in [0,T),
\end{eqnarray}
where $\tilde{c}$ is the speed of the TS in Theorem \ref{thm:TS-in-cylinder}.
\end{thm}

\begin{proof}
By \eqref{cond-H-b<0<k}, the mean curvature $H(r,t)$ remains positive in $[0,T_0)$ for some $T_0\leq T$. We first show that all the inequalities hold in this time interval.

Since $H(r,t)>0$ for $t\in [0,T_0)$, the equation \eqref{eq0} or \eqref{E} is a parabolic one. So, for some large $M'_0>0$, the TS $\wt{\Phi}(r)+ \tilde{c}t+ M'_0$ (which exists by $(-b)^{1/\alpha}\sqrt{1+k^2}<kN$) is an upper solution of initial-boundary value problem \eqref{E}-\eqref{BC1} with $u(r,0)=u_0(r)$. Similarly, $\wt{\Phi}(r)+ \tilde{c} t -M'_0$ is a lower solution. (If one worries about the unboundedness of the coefficient $\frac{N-1}{r}$ in the equation \eqref{E}, he/she can use the equation \eqref{eq0} instead of \eqref{E} to proceed the comparison.) The estimates (4.4) then follows from the comparison principle, for some $M_0$ depending on $b,k$ and $\|u_0\|_C$.

We next derive the $C^1$ bound by using the equation \eqref{eq0}.
(The reason we use the equation \eqref{eq0} instead of the radially symmetric form \eqref{E} is to avoid the possible unboundedness of the coefficients  involving $\frac{u_r}{r}$. By further estimates like $u_r \leq K r$, which will be shown below, we see that $\frac{u_r}{r}$ is actually bounded.)  Denote
\begin{equation}\label{def-xi-A}
\xi:= \sqrt{1+|Du|^2} = \sqrt{1+u_r^2},\quad A:= \delta_{ij} -\frac{D_i u D_j u}{\xi^2}.
\end{equation}
Then the equation \eqref{eq0} is
\begin{equation}\label{eq00}
u_t = (H^\alpha +b)\xi = \left( \frac{A}{\xi}D_{ij} u\right)^{\alpha} \xi + b\xi,\quad |x|<1,\ t\in [0,T).
\end{equation}
Denote $w_1 (x,t):= u_{x_1} (x,t)$, then
$$
w_{1t} = \alpha A H^{\alpha -1} D_{ij}w_1 + B_1 DuDw_1 + B_2 D_{ij} u D_i uD_j w_1,\quad |x|<1,\ t\in [0,T_0),
$$
where
$$
B_1 := \frac{-\alpha H^{\alpha-1} \xi^2 A D_{ij} u  + 2\alpha H^{\alpha-1} D_i u D_j u D_{ij}u +(H^\alpha +b)\xi^3}{\xi^4},\quad B_{2}:=  - \frac{2\alpha H^{\alpha-1}}{\xi^2}.
$$
The boundary conditions \eqref{BC1} imply that
$$
w_1 (x,t)= \left. u_r \frac{x_1}{r}\right|_{r=1} = kx_1,\quad |x|=1,\ t\in [0,T_0).
$$
Then, using the maximum principle for $w_1$ we conclude that
$$
|w_1(x,t)| \leq M_1 := \|u'_0\|_C ,\quad |x|\leq 1,\ t\in [0,T_0).
$$
Hence,
\begin{equation}\label{bound-ux1}
|u_r(r,t)| = |w_1 (r,0,\cdots, 0,t)|\leq M_1,\quad r\in [0,1],\ t\in [0,T_0).
\end{equation}
By $H>0$ we have $H(0,t)=Nu_{rr}(0,t)>0$. Hence $u_r(r,t)>0$ for small $r>0$. Though we can not derive $u_{rr}>0$, we do have the conclusion $u_r>0$ for $r\in (0,1], t\in [0,T_0)$. Otherwise, assume, for some $t_1\in [0,T_0)$, there exists $r_1\in (0,1)$ such that it is the smallest positive zero of $u_r(\cdot,t_1)$, then
$u_{rr}(r_1, t_1)\leq 0$, and so
$$
H(r_1,t_1) = \left. \frac{u_{rr}}{(1+u^2_r)^{3/2}} +\frac{(N-1)u_r}{r\sqrt{1+u^2_r}} \right|_{(r_1, t_1)} \leq 0,
$$
contradicts our assumption. Hence, the estimate \eqref{bound-ux1} can be improved to (4.5).

Similarly, if we write $w (x,t):= u_t(x,t)$, then
$$
w_t = \alpha A H^{\alpha -1} D_{ij}w + B_3 DuDw + B_2 D_{ij} u D_i uD_j w,\quad |x|<1,\ t\in [0,T_0),
$$
where $B_2$ is as above and
$$
B_3 := \frac{\alpha H^{\alpha-1} \xi^2 AD_{ij} u + 2\alpha H^{\alpha-1} D_i u D_j u D_{ij}u +(H^\alpha +b)\xi^3}{\xi^4}.
$$
The boundary condition \eqref{bdry0} implies that
$$
D w \cdot \nu =0,\quad |x|=1,\ t\in [0,T_0),
$$
since $u$ is radially symmetric. Then, using the maximum principle for $w$ we have
\begin{equation}\label{bounds-ut-0}
\min u_t(r,0) \leq w(x,t) = u_t(x,t) \leq \max u_t(r,0),\quad r\in [0,1],\ t\in [0,T_0).
\end{equation}
Denote
$$
\xi^0_0 := \max \xi(r,0) = \sqrt{1+[\max u'_0(r)]^2},\quad \xi^0 := \sqrt{1+M_1^2} \geq \xi(r,t).
$$
Then we have
\begin{equation}\label{bounds-ut}
0< V_* := \min H^\alpha(r,0)+b \leq u_t= (H^\alpha +b)\xi \leq V^* := [\max H^\alpha(r,0)+b]\xi^0_0
\end{equation}
for $r\in [0,1]$ and $t\in [0,T_0)$, and so
\begin{equation}\label{bounds-H}
H_* := \left( \frac{V_*}{\xi^0} -b\right)^{1/\alpha} \leq H(r,t)\leq H^*:= \left( V^* -b\right)^{1/\alpha},\quad r\in [0,1],\ t\in [0,T_0).
\end{equation}
This proves (4.7) (note that $H_*>0$ due to $b\leq 0$).

Furthermore, we can use the second inequality in \eqref{bounds-H} to conclude that
$$
u_{rr} \leq H^* \xi^3 -\frac{(N-1)u_r \xi^2}{r}\leq H^* (\xi^0)^3,\quad r\in [0,1],\ t\in [0,T_0).
$$
This implies that
$$
u_r(r,t) = \int_0^r u_{rr}(r,t) \leq H^* (\xi^0)^3 r,\quad r\in [0,1],\ t\in [0,T_0).
$$
On the one hand, this implies that
\begin{equation}\label{bound-ur-r}
0\leq \frac{u_r}{r}\leq H^*(\xi^0)^3,\quad r\in [0,1],\ t\in [0,T_0),
\end{equation}
and so the radially symmetric equation \eqref{E} in fact has no singularity at $r=0$. On the other hand, by the first inequality in \eqref{bounds-H} we have
\begin{equation}\label{lower-bound-urr}
u_{rr} \geq H_* \xi^3 -\frac{(N-1)u_r \xi^2}{r}\geq H_* - (N-1)H^* (\xi^0)^5,\quad r\in [0,1],\ t\in [0,T_0).
\end{equation}
This proves the bound in (4.8) for $u_{rr}$.

By the above results, especially, the positive lower and upper bounds for $H$, we see that the equation is a uniform parabolic one till the time $T_0$. So the interval $[0,T_0)$ where $H$ remains to be positive can be extended to $[0,T)$. All the estimates we have proved remain valid in $[0,T)$.

This prove the theorem.
\end{proof}

\subsection{Stability of the TS}
Based on the a priori estimate we can prove the following stability result of TS in case $b\leq 0$

\begin{thm}\label{thm:stability-TS-b<0}
Assume $b\leq 0<k$, $(-b)^{1/\alpha}\sqrt{1+k^2}<kN$ and \eqref{cond-H-b<0<k} holds. Then the problem \eqref{E}-\eqref{BC1} with initial data $u(r,0)=u_0(r)$ has a unique global classical solution $u(r,t)$, all the estimate results in the previous theorem hold, and the limit \eqref{asy-stable} holds.
\end{thm}

\begin{proof}
Using the assumption \eqref{cond-H-b<0<k} one can obtain the local existence for the classical solution $u$.  Then using the a priori estimates given in the previous theorem one can obtain the global solution by the parabolic theory.

Next, we study the asymptotic stability of TS, or, the convergence of $u$ to the TS. Denote
$$
\tilde{u}(r,t):= u(r,t) - \tilde{c}(b,k)t,\quad r\in [0,1],\ t\geq 0.
$$
By $|\tilde{u}(r,t)|\leq M_0$, the uniform-in-time estimates for $u_r,\ u_{rr}$ and $H$, we can use the theory for fully nonlinear uniform parabolic equation to conclude that, for any given $\beta\in (0,1)$,
$$
\|\tilde{u}(r,t)\|_{C^{2+\beta, 1+\beta/2} ([0,1]\times [t_1, t_2])}\leq C,\quad 0<t_1 <t_2,
$$
for some $C$ independent of $t_1,t_2$.  Therefore, the $\omega$-limit set of $\tilde{u}$ is non-empty, connected, invariant and compact. To finish the proof we only need to show that the $\omega$-limit set is a singleton $\wt{\Phi}(r)+M$ for some $M\in \R$.

We first show that any $\omega$-limit of $\tilde{u}$ is a shift of $\wt{\Phi}(r)$. Assume, for example,
there exist a time sequence $\{t_n\}$ and a function $U(r,t)\in C^{2+\beta,1+\beta/2}([0,1]\times (0,\infty))$ such that
$$
\tilde{u}(r,t_n +t) -U(r,t) \to 0 \mbox{\ \ as\ \ }n\to \infty,
$$
in the topology of $C^{2+\beta', 1+\beta'/2}_{loc}([0,1]\times \R)$ for  $\beta'\in (0,\beta)$. We show that $U(r,t)\equiv \wt{\Phi}(r)+M$ for some $M\in \R$. Otherwise, assume by contradiction that
\begin{equation}\label{U-leq=}
\wt{\Phi}(r)+M_1 \preceq U(r,0) \preceq \wt{\Phi}(r)+M_2,
\end{equation}
for some $M_2 >M_1$. Here, we write $f_1(r)\preceq f_2(r)$ if $f_1(r)\leq f_2(r)$ and \lq\lq equality" does hold at some point. Since both $U$ and $\wt{\Phi}(r)+M_2$ satisfy the equation of $\tilde{u}$ and the boundary condition \eqref{BC1}, by the strong comparison principle we have
$$
U(r,1) < \wt{\Phi}(r)+M_2 - 2\varepsilon,\quad r\in [0,1],
$$
for small $\varepsilon>0$. Since $\tilde{u}(r,t_n +1)\to U(r,1)$ as $n\to \infty$, for large $n_0$, there holds
$$
\tilde{u}(r,t_n +1) <U(r,1)+\varepsilon < \wt{\Phi}(r)+M_2 -\varepsilon,\quad r\in [0,1],\ n\geq n_0.
$$
Using comparison again we have
$$
\tilde{u}(r,t) \leq \wt{\Phi}(r)+M_2 -\varepsilon,\quad r\in [0,1],\ t\geq t_{n_0}+1.
$$
By $\tilde{u}(r,t_n)\to U(r,0)$ we have
$$
U(r,0)\leq \wt{\Phi}(r) +M_2 -\varepsilon,
$$
contradicts the second inequality in \eqref{U-leq=}.  This proves that any $\omega$-limit of $\tilde{u}$ is $\wt{\Phi}(r)+M$ for some $M\in \R$.

Finally, we show that $\omega$-limit set is a singleton, that is, the above $M$ is unique. Otherwise, there exists $\epsilon>0$ (the case $\epsilon<0$ is proved similarly) such that both $\wt{\Phi}(r)+M$ and $\wt{\Phi}(r)+M+2\epsilon$ are $\omega$-limits of $\tilde{u}$. So, there exists large $T$ such that
$$
\|\tilde{u}(r,T) - \wt{\Phi}(r)-M\|_{C^2 ([0,1])}\leq \epsilon,
$$
which, by comparison, implies that
$$
\tilde{u}(r,t) \leq  \wt{\Phi}(r)+ M + \epsilon,\quad r\in [0,1],\ t\geq T.
$$
This, however,  contradicts the assumption that $\wt{\Phi}(r)+M + 2\epsilon$ is also an $\omega$-limit of $\tilde{u}$.

This completes the proof of the theorem.
\end{proof}

\section{Asymptotical Stability of the TSs: the Case $b>0$}
This section deals with the case $b>0$. In this case, the profile of a TS can be a convex one when $b>0,\ k>0$, or a concave one when $b>0>k$. As in the previous section, the key step is to give the {\it uniform positive lower  bound of $H$} in the former case, and the {\it uniform negative upper bound of $H$} in the latter case. However, unlike deriving $H\geq H_*>0$ in \eqref{bounds-H} from \eqref{bounds-ut} due to $b\leq 0$, in the case where  $b>0$ we have to
adopt other approach to derive the positive lower (or negative upper) bound for $H$. We will consider the case $b>0, k>0$ and $b>0>k$ respectively in two subsections. We do not consider the case $b>0=k$, since the profile of the TS in this case is a horizontal plane, whose mean curvature is $0$, and so the corresponding parabolic equation \eqref{E} is a degenerate one. Of course, in case $\alpha=1$, the equation is a quasilinear one and we can consider the case $b>0=k$. Similar argument as in this and the previous sections leads to the result: any solution of \eqref{E}-\eqref{BC1} converges to the TS with flat profile.

\subsection{The case $b>0, k>0$}
As above, we first derive the a priori estimates. Assume $u(r,t)$ is a classical solution in the time interval $[0,T)$, of the problem \eqref{E}-\eqref{BC1} with $k>0$ and with initial data satisfying \eqref{cond-IC}.
We will derive the a priori estimates as that in Theorem \ref{thm:est-b<0}
by several lemmas.

On the $C^1$ estimates we have the following results.

\begin{lem}\label{lem:c>b>0-C1}
Assume $b>0,\ k>0$, and $H(r,t)>0$ in $[0,T_0)$ for some $T_0\leq T$. Then there exist positive constants $M_0, M_1, M_2, V_*, V^*, H^*$, independent of $t$ and $T_0$, such that
\begin{eqnarray*}
|u(r,t) - \tilde{c}(b,k)t|\leq M_0,\quad r\in [0,1],\ t\in [0,T_0);\\
0< u_r(r,t) \leq M_1,\quad r\in (0,1],\ t\in [0,T_0);\\
V_* \leq u_t(r,t) \leq V^*,\quad r\in [0,1],\ t\in [0,T_0);\\
0<  H(r,t) \leq H^*, \quad r\in [0,1],\ t\in [0,T_0);\\
|u_{rr}(r,t)|\leq M_2,\quad r\in [0,1],\ t\in [0,T_0).
\end{eqnarray*}
\end{lem}

\begin{proof}
The proof is similar as that for Theorem \ref{thm:est-b<0}. The main difference is that the lower bound of $H$ is $0$ by the assumption, rather than a positive number as in \eqref{bounds-H} due to $b\leq 0$ in that case. Therefore, the lower bound of $u_{rr}$ in \eqref{lower-bound-urr} is now replaced by
$$
u_{rr} > -\frac{(N-1)u_r \xi^2}{r}\geq - (N-1)H^* (\xi^0)^5,\quad r\in [0,1],\ t\in [0,T_0).
$$
This proves the lemma.
\end{proof}

As in the previous section, the uniform-in-time positive lower bound of $H$, which has not been given in the above lemma, is the key step in our approach. We will study it under the following condition:
\begin{equation}\label{u0''>0}
u''_0(r)>0,\quad r\in [0,1],
\end{equation}
which is a little stronger than the condition $H(r,0)>0$.

First, we use the zero number argument to present a uniform positive lower bound for $H(0,t)$.

\begin{lem}\label{lem:H(0,t)>0}
Assume the hypotheses in the previous lemma hold, and \eqref{u0''>0} holds. Then there exists $\delta_0>0$ depending only on $b, N$ and $u_0$ such that $u_{rr}(0,t)>\delta_0$ for all $t\in [0,T_0)$.
\end{lem}

\begin{proof}
For any $c>b$ with $c-b$ sufficient small, by Theorem \ref{thm:TSc>b>0}, the equation \eqref{E} has a TS $\Phi(r;c,b)+ct$, and
$$
\Phi'(r)<u'_0(r),\quad r\in (0,1].
$$
This is possible since $u''_0(r)>0$ and $\Phi'(r)\to 0$ uniformly in $r\in [0,1]$, as $c\to b+0$. Then for any $d\in \R$, the relationship between $\Phi(r;c,b)+d$ and $u_0(r)$ is one of the following 4 cases:

(1). no contact point;

(2). $\Phi(0;c,b)+d = u_0(0)$ and $\Phi(r;c,b)+d<u_0(r)$ for $r\in (0,1]$;

(3). there is exactly one point $r_0 \in (0,1)$ with
$$
\Phi(r;c,b)+d>u_0(r) \mbox{ for }r\in [0,r_0),\quad
\Phi(r;c,b)+d<u_0(r) \mbox{ for }r\in (r_0,1].
$$

(4). $\Phi(r;c,b)+d>u_0(r)$ for $r\in [0,1)$, $\Phi(1;c,b)+d = u_0(1)$ and
$\Phi'(1;c,b) < u'_0(1)$.

\noindent
Since both $u_1 (x,t):= u(|x|,t)$ and $u_2(x,t):= \Phi(|x|;c,b)+d+ct$ are radially symmetric solutions of the equation \eqref{eq0}, by using the zero number diminishing property as in \cite{Ang, ChPo} we conclude that, for any time $t\in (0,T_0)$, the relationship between $u(r,t)$ and $\Phi(r;c,b)+ct+d$ is also one of the above 4 cases. In particular, there exists a unique $d=d(t)$ such that $\Phi(r;c,b)+ct+d(t)$ is tangent to $u(r,t)$ at $r=0$, and so
$$
u_{rr}(0,t) \geq  \Phi''(0;c,b) = \frac{ (c-b)^{1/\alpha}}{N},\quad t\in [0,T_0).
$$
This proves the lemma.
\end{proof}

For the speed parameter $c$ in the above proof, we can take another smaller $c_1\in (b,c)$ such that, the TS $\Phi_1(r)+c_1 t$ with $\Phi_1(r):= \Phi(r;c_1,b)$ satisfying
\begin{equation}\label{choice-c1}
u_{rr}(0,t) -\Phi''_1(0) \geq \delta_1:= \frac{(c-b)^{1/\alpha} -(c_1 -b)^{1/\alpha}}{N},\quad t\in [0,T_0),
\end{equation}
\begin{equation}\label{right<k}
\Phi'_1(1)<k,
\end{equation}
and
\begin{equation}\label{choice-c12}
u'_0(r) -\Phi'_1(r) >0, \quad r\in (0,1].
\end{equation}

\begin{lem}\label{lem:ur>Phi'1}
There holds
\begin{equation}\label{ur>Phi'1}
\eta(r,t):= u_r(r,t)-\Phi'_1(r)>0,\quad r\in (0,1],\ t\in [0,T_0).
\end{equation}
\end{lem}

\begin{proof}
We only need to prove the conclusion in $t\in [0,T_0 -\epsilon]$ for any small $\epsilon>0$. In this time interval we have
$$
H(r,t)\geq h_*(\epsilon)>0,\quad r\in [0,1],\ t\in [0,T_0],
$$
and so the equation is a uniform parabolic one. By using the non-radially symmetric equation \eqref{eq0} if necessary we have the parabolic estimate
$$
|u_{rrr}(r,t)|\leq M_3 (T_0, \epsilon),\quad r\in [0,1],\ t\in [0,T_0-\epsilon].
$$
Therefore, by \eqref{choice-c1}, there exists $r_0\in (0,1)$ small such that
$$
u_{rr}(r,t) -\Phi''_1(r) >0,\quad r\in [0,r_0],\ t\in [0,T_0 -\epsilon],
$$
and so
\begin{equation}\label{eta>0-near-0}
\eta(r,t) = u_r(r,t)-\Phi'_1(r) >0,\quad r\in [0,r_0],\ t\in [0,T_0-\epsilon].
\end{equation}

Denote
$$
\xi:= \sqrt{1+u^2_r},\quad H:= \frac{u_{rr}}{\xi^3} +\frac{(N-1)u_r}{r\xi}
,\quad \xi_1 := \sqrt{1+(\Phi'_1)^2},\quad H_1 := \frac{\Phi''_1}{\xi_1^3} +\frac{(N-1)\Phi'_1}{\xi_1}.
$$
We will use $A_1, A_2, \cdots$ to denote continuous functions of $r, u_r, u_{rr}, H$ and/or $\Phi'_1,\Phi''_1, H_1$. Differentiating the equations
$$
u_t = (H^\alpha +b)\xi,\quad c_1 = (H_1^{\alpha}+b)\xi_1
$$
respectively in $r$ we have
\begin{equation}\label{eq-ur}
 u_{rt} = \frac{\alpha H^{\alpha -1}}{\xi^2} u_{rrr} + A_1 u_{rr} - \frac{(N-1)\alpha H^{\alpha-1}}{r^2} u_r
\end{equation}
and
$$
0= \frac{\alpha H_1^{\alpha -1}}{\xi_1^2} \Phi'''_1 + A_2  \Phi''_1 + A_3  \Phi'_1.
$$
Subtracting the second equation from the first one we have
$$
\eta_t = \frac{\alpha H^{\alpha-1}}{\xi^2} \eta_{rr} + A_4 \eta_r + A_5 \eta,\quad r\in (0,1),\ t\in [0,T_0 - \epsilon].
$$

Set
$$
D_1 := [r_0, 1]\times [0,T_0 -\epsilon].
$$
Since all of the functions $\frac{1}{r},\ u_r,\ u_{rr}, H^{\alpha-1}, \Phi'_1,\Phi''_1, H_1^{\alpha-1}$ are bounded in $D_1$, so are the coefficients in the equation of $\eta$. On the other hand, $\eta(r,t)>0$ on the parabolic boundary of $D_1$ by \eqref{eta>0-near-0}, \eqref{right<k} and \eqref{choice-c12}. Hence, the maximum principle implies that
$$
\eta(r,t)>0 \mbox{\ \ in\ \ }D_1.
$$
Combining with \eqref{eta>0-near-0} we obtain the conclusion.
\end{proof}

Next, we show that $u_{rr}>0$ in $[0,T_0)$.

\begin{lem}\label{lem:urr>0}
Assume the hypotheses in Lemma \ref{lem:c>b>0-C1} hold and \eqref{u0''>0} holds. Then $u_{rr}(r,t)>0$ for all $r\in [0,1], t\in [0,T_0)$.
\end{lem}

\begin{proof}
Since $u''_0(r)>0$, $u_{rr}(r,t)>0$ for small $t>0$ by continuity. Assume, by contradiction, that, for some $T_1\in (0,T_0)$ and some $r_1\in (0,1]$, there hold
\begin{equation}\label{urr>0=0}
u_{rr}(r,t)>0 \mbox{ for } r\in [0,1], t\in [0,T_1),\quad u_{rr}(r_1, T_1)=0.
\end{equation}

Differentiating the equation \eqref{eq-ur} again with respect to $r$ we have,
with $\zeta:= u_{rr}$
$$
\zeta_t = \frac{\alpha H^{\alpha-1}}{\xi^2} \zeta_{rr} + A_6 \zeta_r +A_7 \zeta
+ \frac{\alpha^2 (N-1)^2 u^2_r H^{\alpha -2}}{r\xi^4}
>\frac{\alpha H^{\alpha-1}}{\xi^2} \zeta_{rr} + A_6 \zeta_r +A_7 \zeta,
$$
where, as above, $A_6, A_7$ are continuous functions of $\frac1r, u_r, \zeta, \zeta_r$, which may be unbounded when $\frac1r$ is involved.

Since $H>0$ in $[0,T_0)$, it has a positive lower bound in $[0,T_1]$. Hence, the equations of $u,\eta,\zeta$ are all uniform parabolic ones, and so
$$
|\zeta_r (r,t)| =| u_{rrr}(r,t)| \leq M'_3,\quad r\in [0,1],\ t\in [0,T_1].
$$
Combining with \eqref{choice-c1} we see that for some $\tilde{r}_0\in (0,1)$ small,
$$
\zeta(r,t)=u_{rr}(r,t)>0,\quad r\in [0,\tilde{r}_0],\ t\in [0,T_1].
$$

If the zero $r_1$ of $u_{rr}(\cdot,T_1)$ is in $(\tilde{r}_0,1)$, then we can use the strong maximum principle for $\zeta$ in the domain $[\tilde{r}_0,1]\times [0,T_1]$ to get a contradiction. If $r_1=1$, then we can use the Hopf lemma in the same domain to conclude that
$$
\zeta_r (1,T_1) =u_{rrr}(1,T_1)<0.
$$
Substituting this inequality and $u_r(1,t)\equiv k,\ u_{rr}(1,T_1)=0$ into the equation \eqref{eq-ur} we also have a contradiction to $u_{rt}(1,t)\equiv 0$.

Hence, \eqref{urr>0=0} can not be true. This proves the lemma.
\end{proof}

Combining the above lemmas we obtain a positive lower bound for $H$.

\begin{prop}\label{prop:H>delta}
Assume the hypotheses in Lemma \ref{lem:c>b>0-C1} hold and \eqref{u0''>0} holds. Then, there exists $H_*>0$, independent of $t$ and $T_0$, such that
$$
H(r,t)\geq H_*,\quad r\in [0,1],\ t\in [0,T_0).
$$
\end{prop}

\begin{proof}
For all $r\in [0,1]$ and $t\in [0,T_0)$, by Lemmas \ref{lem:urr>0} and \ref{lem:ur>Phi'1} we have
$$
H = \frac{u_{rr}}{\xi^3} +\frac{(N-1)u_r}{r\xi} \geq \frac{(N-1)u_r}{r\xi}
\geq \frac{N-1}{\sqrt{1+k^2}} \cdot \frac{\Phi'_1(r)}{r} \geq H_* := \frac{N-1}{\sqrt{1+k^2}} \cdot \min\limits_{r\in [0,1]} \frac{\Phi'_1(r)}{r}.
$$
This proves the conclusion.
\end{proof}

With this uniform positive lower bound of $H$ in hand, we can prove the following result in a similar way as in the previous section.

\begin{thm}\label{thm:stabilty-TS-c>b>0}
Assume $b>0,\ k>0$ and $u''_0(r)>0$. Then the problem \eqref{E}-\eqref{BC1} with initial data $u(r,0)=u_0(r)$ has a unique global classical solution $u(r,t)$, all the estimate results in Lemma \ref{lem:c>b>0-C1} and the positive lower bound of $H$ in Proposition \ref{prop:H>delta} hold for all $t\geq 0$.

Moreover, the TS $\wt{\Phi}(r;k)+ \tilde{c}(b,k)t$ is asymptotically stable in the sense that \eqref{asy-stable} holds.
\end{thm}

\subsection{The case $b>0>k$}

In this case, the TS has a concave profile with negative mean curvature, so  it is natural to consider the solution $u$ of \eqref{E}-\eqref{BC1} with $H<0$.
To derive the convergence of the profile of $u$, as above we need a uniform-in-time negative lower and upper bounds for $H$. In particular, the negative upper bound for $H$ is essential in the approach. For this purpose we require that the initial $u_0(r)$ not only satisfies $H(r,0)<0$ but also
satisfies
\begin{equation}\label{cond-u0-negative}
u''_0(r)<0,\quad r\in [0,1].
\end{equation}

In a similar way as in the previous subsection we can derive the following result.

\begin{thm}\label{thm:b>c>0}
Assume $\alpha=\frac{q}{p}$ for some positive odd number $p$ and $q$, $b>0>k$, and $b^{{1/\alpha}} \sqrt{1+k^2}> -k N$. Assume also that \eqref{cond-u0-negative} and $H^\alpha(r,0)+b>0$. Then
\begin{enumerate}[{\rm (i)}]
\item the problem \eqref{E}-\eqref{BC1} with initial data $u(r,0)=u_0(r)$ has a unique global classical solution $u(r,t)$;
\item there exist positive constants $M_0, M_1, M_2, V_*, V^*, H_*, H^*$, such that
\begin{eqnarray*}
|u(r,t) - \tilde{c}(b,k)t|\leq M_0,\quad r\in [0,1],\ t\geq 0;\\
0> u_r(r,t) \geq - M_1,\quad r\in (0,1],\ t\geq 0;\\
V_* \leq u_t(r,t) \leq V^*,\quad r\in [0,1],\ t\geq 0;\\
- H^* \leq H(r,t) \leq -H_*, \quad r\in [0,1],\ t\geq 0;\\
|u_{rr}(r,t)|\leq M_2,\quad r\in [0,1],\ t\geq 0;
\end{eqnarray*}
\item the TS $\wt{\Phi}(r;k)+\tilde{c}(b,k)t$ is asymptotically stable in the sense that \eqref{asy-stable} holds.
\end{enumerate}
\end{thm}


\end{document}